\documentclass[12pt]{amsart}
\usepackage{amsmath,amsthm,amsfonts,amssymb}
  
\newtheorem{theorem}{Theorem}
\newtheorem{lemma}{Lemma}

\newtheorem{prop}{Proposition}

\numberwithin{equation}{section}

\begin{document}

\title[Can the Wave SPDE hit Zero?]{Can the Stochastic Wave Equation with 
Strong Drift Hit Zero?}
\author[Lin and Mueller]{Kevin Lin\and Carl Mueller}
\address{Upstart Network, Inc.}
\urladdr{https://www.upstart.com/}
\address{Dept. of Mathematics
\\University of Rochester
\\Rochester, NY  14627}
\urladdr{http://www.math.rochester.edu/people/faculty/cmlr}
\thanks{Supported by a Simons grant.} 
\begin{abstract}
We study the stochastic wave equation with multiplicative noise and singular 
drift: 
\begin{equation*}
\partial_tu(t,x)=\Delta u(t,x)+u^{-\alpha}(t,x)+g(u(t,x))\dot{W}(t,x)
\end{equation*}
where $x$ lies in the circle $\mathbf{R}/J\mathbf{Z}$ and $u(0,x)>0$.  
We show that 

(i) If $0<\alpha<1$ then with positive probability, $u(t,x)=0$ for some 
$(t,x)$.  

(ii) If $\alpha>3$ then with probability one, $u(t,x)\ne0$ for all $(t,x)$.  
\end{abstract}
\maketitle

\section{Introduction}

\label{section:introduction}
One of the classic questions about stochastic processes is whether they can 
hit a given set.  That is, for a process $X_t$ taking values in a space $S$, 
and for $A\subset S$, do we have
\begin{equation*}
\mathbf{P}(X_t\in A\text{ for some $t$})>0.
\end{equation*}
For example, consider the Bessel process $R_t$ with parameter $n$, which 
satisfies
\begin{equation*}
dR=\frac{n-1}{2R}dt+dW
\end{equation*}
where $W(t)$ is a one-dimensional Brownian motion and we assume that $R_0>0$.  
It is well known that if we allow $n$ to take nonnegative real values, then 
$R_t$ can hit 0 iff $n<2$.  For Markov processes such as $R_t$, harmonic 
functions and potential theory are powerful tools which have led to rather 
complete answers to such questions; see \cite{mper10} or most other 
books in Markov processes.  

For stochastic partial differential equations (SPDE), potential theory 
becomes less tractible due to the infinite-dimensional state space of 
solutions, and hitting questions have not been as thoroughly studied.  To be 
specific, solutions $u(t,x)$ usually depend on a time parameter $t$ and a 
spatial parameter $x$.  So for a fixed time $t$, the solution $u(t,x)$ is a 
function of $x$, and the state space of the process is an infinite 
dimensional function space.  

Nonetheless, hitting questions have been studied for certain SPDE, see
\cite{DKN07,DS10,DS15,mt01,NV09} among others.  These papers 
deal with the stochastic heat and wave equations either with no drift or with 
well behaved drift.  

As for SPDE analogues of the Bessel process, the only results known to the 
authors are in Mueller \cite{mue98} and Mueller and Pardoux \cite{mp99}.  
Here we assume that $u(t,x)$ is scalar valued, and as before $t>0$.  But now 
we let $x$ lie in the unit circle $[0,1]$ with endpoints identified.  We 
also assume that $u(0,x)$ is continuous and strictly positive.  Here 
and throughout the paper we write $\dot{W}(t,x)$ for two-parameter white 
noise.  
Suppose $u$ satisfies the following SPDE.  
\begin{equation*}
\partial_tu(t,x)=\Delta u(t,x)+u^{-\alpha}(t,x)+g(u(t,x))\dot{W}(t,x)
\end{equation*}
where there exist constants $0<c_0<C_0<\infty$ for which 
$c_0\leq g(u)\leq C_0$ for all values of $u$.  Let $\tau$ be the first time 
at which $u$ hits 0, and let $\tau=\infty$ if $u$ does not hit 0.  Then 
$\mathbf{P}(\tau<\infty)>0$ if $\alpha<3$, see \cite{mue98} 
Corollary 1.1.  Also, $\mathbf{P}(\tau<\infty)=0$ if $\alpha>3$,
see Theorem 1 of [MP99].

The situation for vector-valued solutions $u(t,x)$ of the stochastic heat 
equation is unclear.  Indeed, the curve $x\to u(t,x)$ may wind around 0, 
and perhaps then $u$ will contract to 0 in cases where it would ordinarily 
stay away from 0.

The purpose of this paper is to study hitting question for the stochastic 
wave equation with scalar solutions and with strong drift.  As is well 
known, there are crucial differences between the heat and wave equations. 
For example, the heat equation satisfies a maximum principle while the wave 
equation does not.  The same holds for the comparison principle, 
which states that if the stochastic heat equation has two solutions with the 
first solution initially larger than the second, then the first solution 
will almost surely remain larger than the second as time goes on.  So while 
certain arguments from the heat equation case carry over, new ideas are 
required.  

Here is the setup for our problem.  Again, we let $t\geq0$, and $x$ lies in 
the circle 
\begin{equation*}
\mathbf{I}=[0,J]
\end{equation*}
with endpoints identified.  We study scalar-valued 
solutions $u(t,x)$ to the following equation.  
\begin{align}
\label{eq:stoch-wave}
\partial_t^2 u(t,x)
 &=\Delta u(t,x)+u^{-\alpha}(t,x)+g(u(t,x))\dot{W(t,x)}  \\
u(0,x)&=u_0(x) \nonumber\\
\partial_tu(0,x)&=u_1(x).   \nonumber
\end{align}
As usual, $u$ and our two-parameter white noise $\dot{W}$ depend on a random 
parameter $\omega$ which we suppress.  As for $x$ taking values in 
higher-dimensional spaces, it is well known that (\ref{eq:stoch-wave}) is 
well-posed only in one spatial dimensions.  Indeed, in 
two or more spatial dimensions we would expect that the solution $u$ only 
exists as a generalized function, but then it is hard to give meaning to 
nonlinear terms such as $u^{-\alpha}$ or $g(u)$.  

Next, we define the first time that $u$ hits 0.  Let
\begin{equation*}
\tau_\infty=\inf\Big\{t>0: \inf_{0\leq s<t}\inf_{x\in\mathbf{I}}u(t,x)=0\Big\}
\end{equation*}
and let $\tau_\infty=\infty$ if the set in the above definition is empty.  

Before stating our main theorems, we give some assumptions.  

\medskip
\noindent
\textbf{Assumptions}
\begin{enumerate}
\item[(i)] $u_0$ is H\"older continuous of order $1/2$ on $\mathbf{I}$.  
\item[(ii)] There exist constants $0<c_0<C_0<\infty$ such that 
$c_0\leq u_0(x)\leq C_0$ for all $x\in \mathbf{I}$.  
\item[(iii)] $u_1$ is H\"older continuous of order $1/2$ on $\mathbf{I}$ and 
hence bounded.  
\item[(iv)] There exist constants $0<c_g<C_g<\infty$ such that 
$c_g\leq g(y)\leq C_g$ for all $y\in\mathbf{R}$.  
\end{enumerate}
\medskip

Here are our main theorems.  
\begin{theorem}
\label{th:2}
Suppose that $u(t,x)$ satisfies (\ref{eq:stoch-wave}), and that the above 
assumptions hold.  Then $\alpha>3$ implies
\begin{equation*}
\mathbf{P}(\tau_\infty<\infty)=0.
\end{equation*}
That is, $u$ does not hit 0.
\end{theorem}

\begin{theorem}
\label{th:1}
Suppose that $u(t,x)$ satisfies (\ref{eq:stoch-wave}), and that above 
assumptions hold.  Then $0<\alpha<1$ implies
\begin{equation*}
\mathbf{P}(\tau_\infty<\infty)>0.
\end{equation*}
That is, $u$ can hit 0.
\end{theorem}

Here is the plan for the paper.  In Section \ref{sec:technicalities} we give 
a rigorous formulation of (\ref{eq:stoch-wave});  in particular, the 
solution is only defined up to the first time $t$ that $u(t,x)=0$ for some 
$x$, since $u^{-\alpha}(t,x)$ blows up there.  The same is true for the 
stochastic heat equation discussed earlier.  In Section \ref{sec:pf-thm-2} 
we prove Theorem \ref{th:2}, and in Section \ref{sec:pf-thm-1} we prove 
Theorem \ref{th:1}.

Note the gap between $\alpha<1$ and $\alpha>3$.  Since there is no 
comparison principle for the wave equation, we cannot be certain that there 
exists a critical value $\alpha_0$ such that $u$ can hit 0 for 
$\alpha<\alpha_0$ but not for $\alpha>\alpha_0$.  We strongly believe in the 
existence of such a critical value, but we leave the existence and 
identification of $\alpha_0$ as an open problem.  

\section{Technicalities}
\label{sec:technicalities}

\subsection{Rigorous Formulation of the Wave SPDE}

For the most part we follow Walsh \cite{wal86} although we could also use the 
formulation found in Da Prato and Zazbczyk \cite{dz92}.

First we recall the definition the one-dimensional wave kernel on $x\in\mathbf{R}$.  
\begin{equation*}
S(t,x)=\frac{1}{2}\mathbf{1}(|x|\leq t)
\end{equation*}
See \cite{eva98} for this classical material.  If we regard $S(t,x)$ as a 
Schwartz distribution, then for $t\geq0$ we can write 
\begin{equation*}
\partial_tS(t,x)=\frac{1}{2}\delta(x-t)+\frac{1}{2}\delta(x+t).
\end{equation*}
From now on, we interpret such expressions as Schwartz distributions.  

Now we switch to the circle $x\in\mathbf{I}$, as defined earlier.  
It is also a classical result that for $x\in\mathbf{I}$, the wave kernel 
$S_\mathbf{I}$ 
and its time derivative are given by
\begin{align*}
S_\mathbf{I}(t,x)&=\sum_{n\in\mathbf{Z}}S(t,x+nJ)  \\
\partial_tS_\mathbf{I}(t,x)&=\frac{1}{2}\sum_{n\in\mathbf{Z}}\Big(\delta(nJ+x-t)+\delta(nJ+x+t)\Big).
\end{align*}
Again, we regard $\partial_tS_\mathbf{I}$ as a Schwartz distribution.  

Let $w(t,x)$ be the solution of the linear deterministic wave equation on $x\in\mathbf{I}$,
with the same initial data as $u$.  That is,
\begin{align*}
\partial_t^2 w(t,x) &=\Delta w(t,x)  \\
w(0,x)&=u_0(x) \\
\partial_tw(0,x)&=u_1(x)  
\end{align*}
with periodic boundary conditions, so that
\begin{align*}
w(t,x) &= \int_{0}^{J}\Big(\partial_tS_\mathbf{I}(t,x-y)u_0(y)
 +S_\mathbf{I}(t,x-y)u_1(y)\Big)dy   \\
&= \frac{1}{2}\int_{0}^{J}\Big(u_0(x-t-y)+u_0(x+t-y)
 +S_\mathbf{I}(t,x-y)u_1(y)\Big)dy 
\end{align*}
where expressions such as $x-y$ and $x-t-y$ are interpreted using arithmetic
modulo J.  We note that by Assumptions (i) and (iii), we can conclude that 
$w(t,x)$ is H\"older continuous of order $1/2$ in $(t,x)$ jointly.  

Using Duhamel's principle, if $u^{-\alpha}$ and 
$g(u(s,y))\dot{W}$ were smooth, we could write
\begin{align}
\label{eq:stoch-wave-int-eq}
u(t,x)=&w(t,x)+\int_{0}^{t}\int_{0}^{J}S_\mathbf{I}(t-s,x-y)u(s,y)^{-\alpha}dyds  \\
& +\int_{0}^{t}\int_{0}^{J}S_\mathbf{I}(t-s,x-y)g(u(s,y))W(dyds).  \nonumber
\end{align}
If $u^{-\alpha}$ had no singularities, we could use this \textit{mild form}
to give rigorous meaning to (\ref{eq:stoch-wave}), where we define
final double integral using Walsh's theory of martingale measures, see 
\cite{wal86}.  One could also use the Hilbert space theory given in Da Prato 
and Zabczyk \cite{dz92}.

To deal with the singularity of $u^{-\alpha}$, we use truncation and then 
take the limit as the truncation is removed.  For $N=1,2,\ldots$ define 
$u_N(t,x)$ as the solution of
\begin{align}
\label{eq:stoch-wave-truncated}
u_N(t,x)
=&w(t,x)+\int_{0}^{t}\int_{0}^{J}S_\mathbf{I}(t-s,x-y)
 \Big[u_N(s,y)\vee(1/N)\Big]^{-\alpha}dyds   \nonumber\\
&+ \int_{0}^{t}\int_{0}^{J}S_\mathbf{I}(t-s,x-y)g(u_N(s,y))W(dyds).
\end{align}
Here $a\vee b=\max(a,b)$.
Note that if $\alpha>0$, then $[u\vee(1/N)]^{-\alpha}$ is a Lipschitz 
function of $u$.  It is well known that SPDE such as 
(\ref{eq:stoch-wave-truncated}) with Lipschitz coefficients have unique strong 
solutions valid for all time, see \cite{wal86}, Chapter III.  It follows for 
each $N=1,2,\ldots$ that (\ref{eq:stoch-wave-truncated}) has a 
unique strong solution $u_N$ valid for all $t\geq0, x\in[0,J]$.  

Now let
\begin{equation*}
\tau_N=\inf\Big\{t>0: \inf_{x\in[0,J]}u_N(t,x)\leq1/N\Big\}.
\end{equation*}
From the definition, we see that almost surely
\begin{equation*}
u_{N_1}(t,x)=u_{N_2}(t,x)
\end{equation*}
for all $t\in[0,\tau_{N_1}\wedge\tau_{N_2})$ and $x\in\mathbf{I}$.  Here 
$a\wedge b=\min(a,b)$.  It also follows that 
$\tau_1\leq\tau_2\leq\cdots$ and so we can almost surely define
\begin{equation*}
\tau=\sup_N\tau_N.  
\end{equation*}
We allow the possibility that $\tau=\infty$.  Note that this definition 
of $\tau$ is consistent with the definition given in the introduction.  

So, for $t<\tau$ and $x\in\mathbf{I}$, we can define 
\begin{equation*}
u(t,x)=\lim_{N\to\infty}u_N(t,x)
\end{equation*}
since for $t<\tau$ and $x\in\mathbf{I}$ the sequence 
$u_1(t,x),u_2(t,x),\ldots$ does not vary with $N$ after a finite number of 
terms.  It follows that $u(t,x)$ satisfies (\ref{eq:stoch-wave-int-eq}) for 
$0\leq t<\tau$.  

Finally, we define $u(t,x)$ for all times $t$ by defining 
\begin{equation*}
u(t,x)=\mathbf{\Delta}
\end{equation*}
for $t\geq\tau$.  Here $\mathbf{\Delta}$ is a cemetary state.  

\subsection{Multi-parameter Girsanov Theorem}

The proof of Theorem \ref{th:1} is based on Girsanov's theorem for 
two-parameter white noise.  This approach was used earlier in Mueller and 
Pardoux \cite{mp99} for the stochastic heat equation, but we need to do some 
work to adapt the argument to the stochastic wave equation.  Girsanov's 
theorem will allow us to remove the drift from our equation 
(\ref{eq:stoch-wave}), at least up to time $\tau$.  If this Girsanov 
transformation gives us an absolutely continuous change of probability 
measure, then we only need to verify that the stochastic wave equation 
(\ref{eq:stoch-wave}) without the drift has a positive probability of 
hitting 0.

Assume that our white noise 
$\dot{W}(t,x)$ and hence also $u(t,x)$ is defined on a probability space 
$(\Omega,\mathcal{F},\mathbf{P})$.  As in Walsh \cite{wal86}, we define 
$\dot{W}(t,x)$ in terms of a random set function $W(A,\omega)$ on measurable 
sets $A\subset[0,\infty)\times\mathbf{I}$.  Let $(\mathcal{F}_t)_{t\geq0}$ 
be the filtration defined by
\begin{equation*}
\mathcal{F}_t=\sigma(W(A): A\subset[0,t]\times\mathbf{I}).
\end{equation*}

Nualart and Pardoux \cite{np94} give the following version of Girsanov's 
theorem.  
\begin{theorem} 
\label{th:girsanov}
Let $T>0$ be a given constant, and define the probability measure 
$\mathbf{P}_T$ to be $\mathbf{P}$ restricted to sets in $\mathcal{F}_T$. 
Suppose that $W$ is a space-time white
noise random measure on $[0,T]\times\mathbf{R}$ with respect to 
$\mathbf{P}_T$, and that $h(t,x)$ is a predictable process such that the 
exponential process
\begin{equation*}
\mathcal{E}_h(t)=\exp{\left(\int_0^t\int_\mathbf{R}h(s,y)W(dyds)
 -\frac{1}{2}\int_0^t\int_\mathbf{R}h(s,y)^2dyds\right)}
\end{equation*}
is a martingale for $t\in[0,T]$. Then the measure
\begin{equation} 
\label{eq:girsanov_measure}
\tilde{W}(dx dt) = W(dx dt) - h(t, x) \ dx dt
\end{equation}
is a space-time white noise random measure on $[0,T]\times\mathbf{R}$ with 
respect to the probability measure $\mathbf{Q}_T$, where $\mathbf{Q}_T$ and 
$\mathbf{P}_T$ are 
mutually absolutely continuous and
\begin{equation} 
\label{eq:girsanov_derivative}
d\mathbf{Q}_T = \mathcal{E}_h(T) \ d\mathbf{P}_T.
\end{equation}
\end{theorem}
We recall Novikov's sufficient condition for $\mathcal{E}_h(t)$ to be a 
martingale.  
\begin{prop}
\label{prop:girsanov}
Let $h(t,x)$ be a predictable process with respect to the filtration 
$(\mathcal{F}_t)_{t\in[0,T]}$. If
\begin{equation} 
\label{eq:novikov}
\mathbf{E}\left[\exp{\left(\frac{1}{2}\int_0^T\int_\mathbf{R}h(s,y)^2dyds\right)}\right]<\infty
\end{equation}
then $\mathcal{E}_h(t)$ is a uniformly integrable $\mathcal{F}_t$-martingale 
for $0\leq t\leq T$.
\end{prop}

\subsection{H\"older continuity of the stochastic convolution}

For an almost surely bounded predictable process $\rho(t,x)$, we define the 
stochastic convolution as follows.  
\begin{equation*}
N_\rho(t,x)=\int_{0}^{t}\int_{0}^{J}S_\mathbf{I}(t-s,x-y)\rho(s,y)W(dyds).
\end{equation*}
Note that the double integral in (\ref{eq:stoch-wave-int-eq}) is equal to 
$N_{g(u)}(t,x)$ for $t<\tau$.  We conveniently define 
$g(\mathbf{\Delta})=0$, so that $N_{g(u)}(t,x)$ is defined for 
all time.  

The proofs of both main theorems depend on the H\"older continuity of 
$N_{g(u)}(t,x)$.  Although such results are common in the SPDE 
literature, unfortunately we could not find the exact result we needed.  
So for completeness, we state it here.  

\begin{theorem}
\label{th:holder}
Let $\rho(t,x)$ be an almost surely bounded predictable process.  
For any $T>0$ and $\beta<1/2$, there exists a random variable $Y$ with 
finite expectation, with $\mathbf{E}|Y|$ depending only on $\beta$ and $T$, such that
\begin{equation}
\label{eq:holder}
\left|N_\rho(t+h,x+k)-N_\rho(t,x)\right|\leq Y\left(h^\beta+k^\beta\right)
\end{equation}
almost surely for all $h,k$ where $t,t+h\in[0,T]$.
\end{theorem}

We will prove Theorem \ref{th:holder} in the appendix.  

\section{Proof of Theorem \ref{th:2}}
\label{sec:pf-thm-2}

\subsection{Outline and Preliminaries}
We write the mild solution to (\ref{eq:stoch-wave}) in the following form:
\begin{equation} 
\label{eq:wave_drifted_sol_short}
u(t,x)=V_u(t,x)+D_u(t,x)+N_u(t,x)
\end{equation}
where
\begin{align*}
V_u(t,x)&=\frac{1}{2}\big(u_0(x+t)+u_0(x-t)\big)+\int_0^Ju_1(y)S_{\mathbf{I}}(t,x-y)dy\\
D_u(t,x)&=\int_0^t\int_0^Ju(s,y)^{-\alpha}S_{\mathbf{I}}(t-s,x-y)dyds\\
N_u(t,x)&=\int_0^t\int_0^Jg(u(s,y))S_{\mathbf{I}}(t-s,x-y)W(dyds).
\end{align*}

We will prove Theorem \ref{th:2} by contradiction. First, we assume that 
$\tau<\infty$ with positive probability. Then, on the sample paths where 
this is the case (i.e., all $u(\omega)$ such that $\tau(\omega)<\infty$),
we go backwards in time from where $u$ hits zero. The upward drift
term $D_u(t,x)$ will then push downwards, since we are going backwards in 
time. We show that this downward push must overwhelm the modulus of 
continuity of the $N_u(t,x)$ term, implying the existence of another time 
$\tau_1<\tau$ such that $ u $ hits zero at $\tau_1$.
However, this contradicts the minimality of $\tau$, thus proving the theorem.

\subsection{A Regularity Lemma}
Let $\mathbf{A}=\{\tau<\infty\}$. By assumption, $\mathbf{P}(\mathbf{A})>0 $. 
We then show the following:

\begin{lemma} 
\label{lem:upper_holder}
On the event $\mathbf{A}$, $V_u(t,x)+N_u(t,x)$ is almost surely $\beta$-H\"older 
continuous on $[0,\tau)\times[0,J]$ for any $\beta<1/2$. The H\"older constant 
is a random variable depending only on $\beta$ and $\omega$.
\end{lemma}

\begin{proof}
Let $\beta<1/2$ be given. Then by (\ref{eq:holder}) we know that $N_u(t,x)$ 
is almost surely $\beta$-H\"older continuous on $[0,\tau)\times[0,J]$, with 
random H\"older constant $Y$ depending only on $\beta$ and $\tau$. Since $u_1$ 
is continuous on $\mathbf{I}$, the Riemann integral
\begin{equation*}
\int_0^Ju_1(y)S_{\mathbf{I}}(t,x-y)dy=\int_{x-t}^{x+t}u_1(y)dy
\end{equation*}
is jointly differentiable (and thus $\beta$-H\"older continuous) on 
$(t,x)\in[0,\tau)\times[0,J]$ as well. Finally, from assumption, $u_0$ is 
$\beta$-H\"older continuous on $[0,J]$, so it follows that
$\frac{1}{2}(u_0(x+t)+u_0(x-t))$ is continuous as well.

Thus $V_u(t,x)+N_u(t,x)$ is almost surely $\beta$-H\"older continuous on 
$[0,\tau)\times[0,J]$. As the H\"older constant of $V_u$ depends only on 
$u_0$ and $u_1$, the H\"older constant of $V_u+N_u$ is a random variable 
depending only on $\beta$.
\end{proof}

\subsection{The Backwards Light Cone}

Given $(t,x)\in\mathbf{R}_+\times\mathbf{R}$, define the \textit{backwards 
light cone} as 
\begin{equation*}
\mathbf{L}(t,x)=\left\{(s,y):\left|x-y\right|<t-s\right\}.
\end{equation*}
Note that the light cone cannot include points $(s,y)$ for which $s>t$.  
It follows that $ D_u(t, x) $ can be rewritten as
\begin{equation} 
\label{eq:upper_d}
\begin{split}
D_u(t,x)&=\int_0^t\int_0^Ju(s,y)^{-\alpha}S_{\mathbf{I}}(t-s,x-y)dyds\\
&=\int_0^t\int_{\mathbb{R}}u\left(s,y^*\right)^{-\alpha}S(t-s,x-y)dyds\\
&=\iint_{\mathbf{L}(t,x)}u\left(s,y^*\right)^{-\alpha}dyds
\end{split}
\end{equation}
where
\begin{equation} 
\label{eq:upper_wrap}
y^* = y \mod{J}.
\end{equation}
and $y^*\in[0,J]$.  

\begin{lemma} \label{lem:upper_drift}
Let $(t,x)\in[0,\tau)\times[0,J]$. Then for any $(s,y)\in\mathbf{L}(t,x)$, 
we have
\begin{equation*}
D_u(s,y)-D_u(t,x)<0.
\end{equation*}
\end{lemma}

\begin{proof}
Since $u(s,y)>0$ on $[0,\tau)$, using (\ref{eq:upper_d}) the result follows 
from the fact that $\mathbf{L}(s,y)\subsetneq\mathbf{L}(t,x)$.
\end{proof}

\subsection{Theorem \ref{th:2}, Conclusion}
Since $\alpha>3$ by assumption, define $\epsilon\in(0,1/2)$ sufficiently 
small such that
\begin{equation*}
\frac{3-\alpha}{2}+\epsilon(\alpha+1)<0.
\end{equation*}

Using Lemma \ref{lem:upper_holder}, on the event $\mathbf{A}$ we define $Y$ 
to be a (random) $1/2-\epsilon$ H\"older constant of $V_u(t,x)+N_u(t,x)$, 
depending only on $\epsilon$. By our choice of $\epsilon$, the exponent of 
$R$ in the expression
\begin{equation*}
\frac{\pi Y^{-1-\alpha}}{2^{\alpha+2}}R^{\frac{3-\alpha}{2}+\epsilon(\alpha+1)}
\end{equation*}
is negative. Hence, on $\mathbf{A}$ we can pick a sufficiently small random 
$R>0$, depending on $\epsilon$ and $Y$, such that both
\begin{equation} 
\label{eq:upper_r}
\frac{\pi Y^{-1-\alpha}}{2^{\alpha+2}}R^{\frac{3-\alpha}{2}+\epsilon(\alpha+1)}>1
\quad\text{and}\quad R<\frac{\tau}{2}.
\end{equation}

Finally, on $\mathbf{A}$ we pick a random $\delta>0$ sufficiently small such 
that both
\begin{equation} 
\label{eq:upper_delta}
\delta<\min\left(\inf_{x\in[0,J]}u_0(x),YR^{\frac{1}{2}-\epsilon}\right)
\end{equation}
and
\begin{equation} 
\label{eq:upper_delta_tau}
\tau_{\delta}=\inf\left\{t>0:\inf_{x\in[0,J]}u(t,x)<\delta\right\}>\frac{\tau}{2},
\end{equation}
which is possible since since $u(t,x)$ is continuous in $t$ for
$t<\tau$. Here, $\tau{_\delta}$ need not be a stopping time. Note that
$\tau_{\delta}$ is the first time that $u(t,x)$ reaches $\delta$, and that 
by continuity of $u(t,x)$ in $x$, there exists some $x_{\delta}\in[0,J]$ 
such that $u\left(\tau_{\delta},x_{\delta}\right)=\delta$. We define the 
differences
\begin{align*}
\Delta V(t,x)&=V_u(t,x)-V_u\left(\tau_{\delta},x_{\delta}\right)\\
\Delta D(t,x)&=D_u(t,x)-D_u\left(\tau_{\delta},x_{\delta}\right)\\
\Delta N(t,x)&=N_u(t,x)-N_u\left(\tau_{\delta},x_{\delta}\right)
\end{align*}
and for all $(t,x)\in\mathbf{L}(\tau_{\delta},x_{\delta})$, we decompose
\begin{equation} 
\label{eq:upper_decomp}
\begin{split}
u(t,x)&=u(t,x)-u\left(\tau_{\delta},x_{\delta}\right)+\delta\\
&=\Delta V(t,x)+\Delta D(t,x)+\Delta N(t,x)+\delta.
\end{split}
\end{equation}

We recall that by construction,
\begin{equation} 
\label{eq:upper_initial_noise}
\Delta V(t,x) + \Delta N(t,x)
<Y\left|(t,x)-\left(\tau_{\delta},x_{\delta}\right)\right|^{1/2-\epsilon}
\end{equation}
almost surely on $\mathbf{A}$ with $\mathbf{E}\left[Y;\mathbf{A}\right]<\infty$. From 
Lemma \ref{lem:upper_drift}, we find that
\begin{equation*}
\Delta D(t,x)<0
\end{equation*}
almost surely. Hence, for all 
$(t,x)\in\mathbf{L}\left(\tau_{\delta},x_{\delta}\right)$ we obtain the bound
\begin{equation} 
\label{eq:upper_initial}
\begin{split}
u(t,x)&=\Delta V(t,x)+\Delta D(t,x)+\Delta N(t,x)+\delta\\
&< \Delta V(t,x) + \Delta N(t,x) + \delta \\
&<Y\left|(t,x)-\left(\tau_{\delta},x_{\delta}\right)\right|^{1/2-\epsilon}
                + \delta
\end{split}
\end{equation}
almost surely on $\mathbf{A}$. We define the sector
\begin{equation*}
B_R=\left\{(t,x)\in\mathbf{L}\left(\tau_{\delta},x_{\delta}\right)
 :\left|\left(\tau_{\delta},x_{\delta}\right)-(t,x)\right|\leq R\right\},
\end{equation*}
noting from (\ref{eq:upper_r}) and (\ref{eq:upper_delta_tau}) that $t>0$ on 
$B_R$. We then denote the curved part of the boundary of $B_R$ by
\begin{equation*}
\partial B_R=\left\{(t,x)\in B_R:
	\left|\left(\tau_{\delta},x_{\delta}\right)-(t,x)\right|=R\right\}.
\end{equation*}
Then for all $(t,x)\in\partial B_R$, using (\ref{eq:upper_d}), 
(\ref{eq:upper_initial}), and (\ref{eq:upper_delta}) we find that
\begin{equation} 
\label{eq:upper_drift}
\begin{split}
\Delta D(t,x)&=-\iint_{\mathbf{L}\left(\tau_{\delta},x_{\delta}\right)
		\setminus\mathbf{L}(t,x)}u\left(s,y^*\right)^{-\alpha}dyds\\
&\leq-\iint_{B_R}u\left(s,y^*\right)^{-\alpha}dyds\\
&\leq-\left|B_R\right|\left(YR^{\frac{1}{2}-\epsilon}+\delta\right)^{-\alpha}\\
&<-\left|B_R\right|\left(2YR^{\frac{1}{2}-\epsilon}\right)^{-\alpha}\\
&=-\frac{\pi R^2}{2^{\alpha+2}}Y^{-\alpha}R^{-\alpha\left(\frac{1}{2}-\epsilon\right)}\\
&=-\frac{\pi Y^{-\alpha}}{2^{\alpha+2}}R^{2-\left(\frac{1}{2}-\epsilon\right)\alpha}
\end{split}
\end{equation}
on the event $\mathbf{A}$. Recall that on $\partial B_R$,
$\left|(t,x)-\left(\tau_{\delta},x_{\delta}\right)\right|=R$.
Hence from
(\ref{eq:upper_decomp}), (\ref{eq:upper_initial_noise}), and 
(\ref{eq:upper_drift}) we find that for all $(t,x)\in\partial B_R$,
\begin{equation} 
\label{eq:upper_boundary}
\begin{split}
u(t,x)&<YR^{\frac{1}{2}-\epsilon}-\frac{\pi Y^{-\alpha}}{2^{\alpha+2}}
 R^{2-\left(\frac{1}{2}-\epsilon\right)\alpha}\\
&=YR^{\frac{1}{2}-\epsilon}\left(1-\frac{\pi Y^{-1-\alpha}}{2^{\alpha+2}}
 R^{2-\left(\frac{1}{2}-\epsilon\right)\alpha
 -\left(\frac{1}{2}-\epsilon\right)}\right)\\
&=YR^{\frac{1}{2}-\epsilon}\left(1-\frac{\pi Y^{-1-\alpha}}{2^{\alpha+2}}
 R^{\frac{3-\alpha}{2}+\epsilon(\alpha+1)}\right)
\end{split}
\end{equation}
almost surely on $\mathbf{A}$. From (\ref{eq:upper_r}) and 
(\ref{eq:upper_boundary}) it then follows that $u(t,x)<0 $ for all 
$(t,x)\in\partial B_R$, almost surely on $\mathbf{A}$.

Since $\mathbf{P}(\mathbf{A})>0 $ by assumption, the event that $u(t,x)<0$ for all
$(t,x)\in\partial B_R$ occurs with positive probability. However, since 
$R>0$, we know that
$t<\tau_{\delta}<\tau$ for all $(t,x)\in\partial B_R$, which is a 
contradiction, since $\tau$ is defined to be the first hitting time for 
$u(t,x)\leq0$. Hence we conclude that $\mathbf{P}(\mathbf{A})=0$.

This finishes the proof of theorem \ref{th:2}.

\section{Proof of Theorem \ref{th:1}}
\label{sec:pf-thm-1}

\subsection{Equation without the Drift}

Now we use Proposition \ref{prop:girsanov} to prove Theorem \ref{th:1}.  
Consider the stochastic wave equation with initial conditions identical to 
(\ref{eq:stoch-wave}) but without drift:
\begin{align}
\label{eq:stoch-wave-no-drift}
\partial_t^2 v(t,x)
 &=\Delta v(t,x)+g(v(t,x))\dot{W(t,x)}  \\
v(0,x)&=u_0(x) \nonumber\\
\partial_tv(0,x)&=u_1(x).   \nonumber
\end{align}
Here $x\in[0,J]$, as before.  
Since there are no singular terms in (\ref{eq:stoch-wave-no-drift}), we can 
give this equation rigorous meaning using the mild form:
\begin{equation}
\label{eq:wave_undrifted_sol}
v(t,x)=w(t,x)+\int_{0}^{t}\int_{0}^{J}S_\mathbf{I}(t-s,x-y)g(v(s,y))W(dyds).  
\end{equation}
where $w(t,x)$ is as before, the solution to the deterministic wave equation.  

First we verify that $v(t,x)$ can hit 0.
\begin{lemma}
\label{lem:v-hit-0}
Suppose that $v(t,x)$ is a solution to (\ref{eq:wave_undrifted_sol}).  Then
\begin{equation*}
\mathbf{P}(v(t,x)=0 \text{ for some $t>0,\, x\in[0,J]$})>0.
\end{equation*}
\end{lemma}

\begin{proof}
Let $V(t)=\int_0^Jv(t,x)dx$.  By the almost sure continuity of 
$v(t,x)$ (see \cite{wal86}) Chapter III, it suffices to show that 
\begin{equation}
\label{eq:v-below-0}
\mathbf{P}\big(V(t)<0\big)>0.
\end{equation}
Since $\int_0^JS_{\mathbf{I}}(t,x-y)dy=t$ by the definition of the 
one-dimensional wave kernel, and since
$\int_0^J\frac{1}{2}\big(u_0(x+t)+u_0(x-t)\big)dx=\int_0^Ju_0(x)dx$, 
\begin{equation*}
V(t)=\int_0^Ju_0(x)dx+t\int_0^Ju_1(x)dx
 +\int_{0}^{t}\int_0^J(t-s)g(v(s,y))W(dyds).
\end{equation*}
Here we have used the stochastic Fubini theorem (see \cite{wal86}, Theorem 
2.6) to change the order of integration in the double integral.  Let us 
define $N_v(t)$ as the double integral:
\begin{equation*}
N_v(t)=\int_{0}^{t}\int_0^J(t-s)g(v(s,y))W(dyds).
\end{equation*}

The question would be easy if $g\equiv1$, as $N_v(t)$ would be a 
Gaussian variable, with a positive probability of taking values below any 
desired level.  Since this is not necessarily the case, we use another 
Girsanov transformation to bound $N_v(t)$ by a Gaussian process.  

Fix $t>0$. Choose $K$ sufficiently large so that
\begin{equation} 
\label{eq:lower_undrifted_k_bound}
\frac{c_gJKt^2}{2}-\int_0^Ju_0(x)dx-t\int_0^Ju_1(x)dx>0.
\end{equation}
Using Theorem \ref{th:girsanov}, we define $ \tilde{W} $ as a $\tilde{\mathbf{P}}$
white noise, where $\mathbf{P}$ and $\tilde{\mathbf{P}}$ are equivalent and
\begin{equation*}
W(dyds)=\tilde{W}(dyds)-Kdyds.
\end{equation*}
Decompose $ N_v(t) = N_v^{(1)}(t) - N_v^{(2)}(t) $, where
\begin{align*}
N_v^{(1)}(t)&=\int_0^t\int_0^J(t-s)g(v(s,y))\tilde{W}(dyds)\\
N_v^{(2)}(t)&=\int_0^t\int_0^J(t-s)g(v(s,y))Kdyds.
\end{align*}
Since $g(v(s,y))$ is bounded below by $c_g>0$, we have:
\begin{equation*}
N_v^{(2)}(t)\geq\frac{c_g JKt^2}{2}.
\end{equation*}
Hence to show (\ref{eq:v-below-0}), it suffices to prove that
\begin{equation*}
\mathbf{P}\left(N_v^{(1)}(t)<\frac{c_gJKt^2}{2}-\int_0^Ju_0(x)dx
		-t\int_0^Ju_1(x)dx\right)>0
\end{equation*}
and since $\mathbf{P}$ and $ \tilde{\mathbf{P}} $ are equivalent, we can show instead that
\begin{equation} 
\label{eq:lower_undrifted_bound}
\tilde{\mathbf{P}}\left(N_v^{(1)}(t)<\frac{c_gJKt^2}{2}-\int_0^Ju_0(x)dx
 -t\int_0^Ju_1(x)dx\right)>0.
\end{equation}
	
We define the process
\begin{equation*}
M_t(r)=\int_0^r\int_0^J(t-s)g(v(s,y))\tilde{W}(dyds).
\end{equation*}
Since $g$ is bounded, $M_t(r)$ is an $\mathcal{F}_r$-martingale in $r$, 
for $r\leq t$. Hence, from Theorem V.1.6 in Revuz and Yor \cite{ry99}, there 
exists a one-dimensional standard Brownian motion $B$ such that
$M_t(r)=B\left(\tau(r)\right)$, where the time change $\tau(r)$ is given by the
predictable process:
\begin{equation*}
\begin{split}
\tau(r)&=\int_0^r\int_0^J(t-s)^2g^2(v(s,y))dyds \\
 &\leq C_g^2\int_0^r\int_0^J(t-s)^2dyds \\
 &=\frac{C_g^2J}{3}t^3-\frac{C_g^2J}{3}(t-r)^3.
\end{split}
\end{equation*}
Then let
\begin{equation*}
L=\frac{C_g^2J}{3}t^3
\end{equation*}
so we have $ \tau(t) \leq L $. Using this, we find that:
\begin{equation*}
N_v^{(1)}(t)=M_t(t)=B\left(\tau(t)\right)\leq\sup_{0\leq q\leq L}B(q).
\end{equation*}
Due to (\ref{eq:lower_undrifted_k_bound}), we can use the reflection principle to find that
\begin{align*}
\begin{split}
\tilde{\mathbf{P}}&\left(\sup_{0\leq q\leq L}B(q)
 \geq\frac{c_gJKt^2}{2}-\int_0^Ju_0(x)dx-t\int_0^Ju_1(x)dx\right) \\
&\leq2\tilde{\mathbf{P}}\left(B(L)\geq\frac{c_gJKt^2}{2}-\int_0^Ju_0(x)dx
	-t\int_0^Ju_1(x)dx\right) \\
&< 1 \qquad\qquad\qquad \text{(since $ B(L) \sim \mathcal{N}(0, L) $)}
\end{split}
\end{align*}
from which (\ref{eq:lower_undrifted_bound}) follows, and the proof of Lemma 
\ref{lem:v-hit-0} is complete.  
\end{proof}

\subsection{Removing the Drift Term}

To finish the proof of Theorem \ref{th:1}, it suffices to show that up to 
the first time $\tau$ that $u$ and $v$ hit 0, these two processes induce 
equivalent probability measures on the canonical paths consisting of 
continuous functions $f(t,x)$ on $[0,\tau(f)]\times[0,J]$.  

Given a (possibly random) function 
$f:[0,\infty)\times[0,J]\rightarrow\mathbf{R}$, define the
hitting times
\begin{align*}
\tau^{(f)}&=\inf\left\{t>0:\inf_{x\in[0,J]}f(t,x)\leq0\right\} \\
\alpha_m^{(f)}&=\inf\left\{t>0:\int_0^t\int_0^Jf(s,x)^{-2\alpha}dxds>m\right\}
\end{align*}
and for a constant $ T > 0 $, let
\begin{equation*}
T_m(f)=\tau^{(f)}\wedge\alpha_m^{(f)}\wedge T.
\end{equation*}
Then define the truncated function $ f^{T_m(f)} $ by:
\begin{equation*}
f^{T_m(f)}(t,x)=f(t,x)\mathbf{1}_{\{t\leq T_m(f)\}}.
\end{equation*}

Let $\mathbf{P}_u^{T_m(u)}$, $\mathbf{P}_v^{T_m(v)}$ be the measures on path space
$\mathcal{C}\left([0,\infty)\times[0,J],\mathbf{R}\right) $ induced by 
$u^{T_m(u)}(t,x)$, $v^{T_m(v)}(t,x)$ respectively, and let
\begin{equation}
h(r) =
\begin{cases}
\label{eq:h-drift}
\frac{r^{-\alpha}}{g(r)}&\text{ if }r\neq0 \\
0&\text{ if }r=0.
\end{cases}
\end{equation}
We then obtain the following Girsanov transformation:
\begin{lemma} \label{lem:lower_girsanov}
For each $ m\in\mathbf{N} $, the measures $\mathbf{P}_u^{T_m(u)}$ and $\mathbf{P}_v^{T_m(v)}$ 
are equivalent, with
\begin{align*}
\frac{d\mathbf{P}_u^{T_m(u)}}{d\mathbf{P}_v^{T_m(v)}}
 =\exp\bigg(\int_0^{T_m(v)}&\int_0^Jh\left(v(t,x)\right)W(dxdt) \\
& -\frac{1}{2}\int_0^{T_m(v)}\int_0^Jh\left(v(t,x)\right)^2dxdt\bigg).
\end{align*}
\end{lemma}

\begin{proof}
First, we note that $h(v^{T_m(v)}(t,x))$ satisfies the Novikov condition 
given in (\ref{eq:novikov}). Then, define the probability measure 
$\mathbf{Q}^{T_m(v)}$ by the derivative
\begin{equation*}
\begin{split}
\frac{d\mathbf{Q}^{T_m(v)}}{d\mathbf{P}_v^{T_m(v)}}
&=\exp\bigg(\int_0^T\int_0^Jh\left(v^{T_m(v)}(t,x)\right)W(dxdt)  \\
&\qquad\qquad-\frac{1}{2}\int_0^T\int_0^Jh\left(v^{T_m(v)}(t,x)\right)^2dxdt\bigg) \\
&=\exp\bigg(\int_0^{T_m(v)}\int_0^Jh\left(v(t,x)\right)W(dxdt)  \\
&\qquad\qquad -\frac{1}{2}\int_0^{T_m(v)}\int_0^Jh\left(v(t,x)\right)^2dxdt\bigg).
\end{split}
\end{equation*}
Then, from Theorem \ref{th:girsanov}, it follows that
\begin{equation*}
\tilde{W}(dxdt)=W(dxdt)-h\left(v^{T_m(v)}(t,x)\right)dxdt
\end{equation*}
is a space-time white noise random measure under $\mathbf{Q}^{T_m(v)}$. Note 
that $\mathbf{Q}^{T_m(v)}$ is the measure on 
$\mathcal{C}\left(\left[0,\infty\right)\times[0,J],\mathbf{R}\right)$ 
induced by $f^{T_m(f)}(t,x)$ where $f(t,x)$ satisfies
\begin{equation*}
\begin{split}
f(t,x)&=\frac{1}{2}\left(u_0(x+t)+u_0(x-t)\right)+\int_0^Ju_1(y)S_{\mathbf{I}}(t,x-y)dy \\
&\quad+\int_0^t\int_0^Jg(f(s,y))S_{\mathbf{I}}(t-s,x-y)W(dyds) \\
&=\frac{1}{2}\left(u_0(x+t)+u_0(x-t)\right)+\int_0^Ju_1(y)S_{\mathbf{I}}(t,x-y)dy \\
&\quad+\int_0^t\int_0^Jf(s,y)^{-\alpha}S_{\mathbf{I}}(t-s,x-y)dxdt \\
&\quad+\int_0^t\int_0^Jg(f(s,y))S_{\mathbf{I}}(t-s,x-y)\tilde{W}(dyds).
\end{split}
\end{equation*}
where the last term is a Walsh integral with respect to the underlying 
measure $\mathbf{Q}^{T_m(v)}$.  However, these are just the paths of $u^{T_m(u)}(t,x)$, 
so the measure $\mathbf{Q}^{T_m(v)}=\mathbf{P}_u^{T_m(u)}$. Then Lemma 
\ref{lem:lower_girsanov} follows.  
\end{proof}

Now we wish to apply Lemma \ref{lem:lower_girsanov} with $h$ as defined in 
\eqref{eq:h-drift}.  This depends on the finiteness of $\alpha_m(f)$ for 
some $m$.  Thus \ref{th:1} follows from the following lemma, which 
we prove by using the regularity of the stochastic wave equation.
\begin{lemma} 
\label{lem:lower_main}
For any constant $T>0$,
\begin{equation} 
\label{eq:lower_main}
\int_0^{\tau^{(v)}\wedge T}\int_0^Jv(t,x)^{-2\alpha}dxdt<\infty
\end{equation}
almost surely.
\end{lemma}

\subsection{Proof of Lemma \ref{lem:lower_main}}
For this entire section, we let $v(t,x)$ be as given in 
(\ref{eq:wave_undrifted_sol}).

\subsubsection{A Rectangular Grid}
For each $ K > 0 $ define the event:
\begin{equation} \label{eq:lower_k}
\mathbf{A}(K):=\left\{\sup_{(t,x)\in[0,T]\times[0,J]}v(t,x)\leq K\right\}.
\end{equation}
Since $(t,x)\mapsto v(t,x)$ is almost surely continuous, the above supremum 
is almost surely finite, so
\begin{equation*}
\lim_{K\rightarrow\infty}\mathbf{P}\left\{\mathbf{A}(K)\right\}=1.
\end{equation*}

We split the interval $ (0, K] $ into dyadic subintervals
\begin{equation}
(0,K]=\bigcup_{n=0}^\infty\left(2^{-n-1}K,2^{-n}K\right]
\end{equation}
and observe that on the event $ \mathbf{A}(K) $,
\begin{equation} 
\label{eq:lower_lebesgue_bound}
\begin{split}
&\int_0^{\tau^{(v)}\wedge T}\int_0^Jv(t,x)^{-2\alpha}dxdt\\
&=\sum_{n=0}^{\infty}\left[\int_0^{\tau^{(v)}\wedge T}\int_0^Jv(t,x)^{-2\alpha}
 1_{\left\{2^{-n-1}K<v(t,x)\leq2^{-n}K\right\}}dxdt\right]\\
&\leq\sum_{n=0}^{\infty}\left[2^{2\alpha(n+1)}K^{-2\alpha}
 \int_0^{\tau^{(v)}\wedge T}\int_0^J1_{\left\{2^{-n-1}K<v(t,x)\leq2^{-n}K\right\}}
 dxdt\right]\\
&=\sum_{n=0}^{\infty}\Big[2^{2\alpha(n+1)}K^{-2\alpha}\\
&\qquad\times\mu\left(\left\{(t,x)\in\left[0,\tau^{(v)}\wedge T\right]\times[0,J]
 :2^{-n-1}K<v(t,x)\leq2^{-n}K\right\}\right)\Big]
\end{split}
\end{equation}
where $\mu$ denotes Lebesgue measure.
	 
Define a constant $\epsilon>0$ such that
\begin{equation*}
0<2\epsilon<1-\alpha
\end{equation*}
and for each $n\in\mathbf{N}$, consider the rectangle
\begin{equation} 
\label{eq:lower_square}
D_n=\left\{(t,x)\in\left[0,\lambda_n\right]\times\left[0,2\lambda_n\right]\right\}
\end{equation}
where
\begin{equation} 
\label{eq:lower_lambda}
\lambda_n=2^{-(1-2\epsilon)n}.
\end{equation}
As far as the optimality of this choice of $\lambda_n$, the real issue is 
why the same factor applies to both $t$ and $x$.  This is because the 
stochastic wave equation with white noise has the same regularity in both $t$ 
and $x$, namely it is H\"older $1/2-\varepsilon$.  So we do not believe 
that the result for $\alpha<1$ can be improved by better choosing $\lambda_n$.

Next, for each $ (t, x) \in D_n $, define the grids of points:
\begin{align*}
\Gamma_n(t,x)&=\Biggl[[0,T]\times[0,J]\Biggr]\bigcap
 \left[\bigcup_{k,\ell\in\mathbb{N}}\Big(t+k\lambda_n,x+2\ell\lambda_n\Big)\right] \\
 \overline{\Gamma}_n(t,x)&=\Biggl[\left[0,\tau^{(v)}\right]\times[0,J]\Biggr]\bigcap\Gamma_n(t, x).
\end{align*}
Let $\#$ denote the number of points in a set, and define the strip
\begin{equation*}
J_n=\left\{(t,x)\in\left[0,\lambda_n\right]\times[0,J]:2^{-n-1}K
		<v(t,x)\leq2^{-n}K\right\}.
\end{equation*}	
Then we have
\begin{multline} 
\label{eq:lower_lebesgue_count}
\mu\left(\left\{(t,x)\in\left[0,\tau^{(v)}\wedge T\right]\times[0,J]:2^{-n-1}K
 <v(t,x)\leq2^{-n}K\right\}\right)\\
\leq\iint_{D_n}\#\left\{(s,y)\in\overline{\Gamma}_n(t,x):v(s,y)
\leq2^{-n}K\right\}dxdt +\mu\left(J_n\right).
\end{multline}

Since $v(t,x)$ is continuous on $[0,T]\times[0,J]$ and 
$\inf_{x\in[0,J]}u_0(x)>0$, we have that
$\mu\left(J_n\right)=0$ for sufficiently large random $n$. Hence,
\begin{equation} 
\label{eq:lower_strip_bound}
\sum_{n=0}^{\infty}\left[2^{2\alpha(n+1)}K^{-2\alpha}\mu\left(J_n\right)\right]<\infty
\end{equation}
almost surely. We now place a bound on
\begin{equation*}
\#\left\{(s,y)\in\overline{\Gamma}_n(t,x):v(s,y)\leq2^{-n}K\right\}
\end{equation*}
in the upcoming lemmas.

\subsubsection{The Shifted Equation}
Let $(t,x)$ be an arbitrary point in $D_n$, as defined in 
(\ref{eq:lower_square}), and let $\theta$ be the time shift operator, 
defined by $\theta_s W(dxdt)=W(dxd(t + s))$.

Then for given $(s,y)\in\Gamma_n(t,x)$, define
\begin{equation*}
s_n^-=s-\lambda_n.
\end{equation*}
Now, we take the approach of considering $W$ as a cylindrical Wiener 
process, as described in \cite{dz92}.
Furthermore, by Theorem 9.15 on page 256 of Da Prato and Zabczyk \cite{dz92},
there is a version of our solution $\Phi_t$ which is a strong Markov process 
with respect to the Brownian filtration $(\mathcal{F}_t)_{t\geq0}$.  

Using the strong Markov property of solutions, we restart the 
equation at time
$s_n^-$:
\begin{multline} 
\label{eq:lower_restart}
v(s,y)=\frac{1}{2}\left(v\left(s_n^-,y+\lambda_n\right)
+v\left(s_n^-,y-\lambda_n\right)\right)  \\
+\int_0^JS_{\mathbf{I}}\left(\lambda_n,y-z\right)\frac{\partial v}{\partial t}\left(s_n^-,z\right)dz\\
+\int_0^{\lambda_n}\int_0^JS_{\mathbf{I}}\left(\lambda_n-r,y-z\right)
g\left(v\left(s_n^-+r,z\right)\right)\theta_{s_n^-}W(dzdr).
\end{multline}
Here, $\frac{\partial v}{\partial t}$ is regarded as a Schwartz distribution.  

We analyze \eqref{eq:lower_restart} term by term. Decompose
\begin{equation*}
v(s,y)=V_n(s,y)+N_n(s,y)+E_n(s,y)
\end{equation*}
where
\begin{align*}
V_n(s,y)=&\frac{1}{2}\left(v\left(s_n^-,y+\lambda_n\right)
 +v\left(s_n^-,y-\lambda_n\right)\right)  \\
&+\int_0^JS_{\mathbf{I}}\left(\lambda_n,y-z\right)\frac{\partial v}{\partial t}\left(s_n^-,z\right)dz\\
N_n(s,y)&=\int_0^{\lambda_n}\int_{\left\{|z-y|\leq\lambda_n\right\}}
 S_{\mathbf{I}}\left(\lambda_n-r,y-z\right)g(v(s_n^-,y))
 \theta_{s_n^-}W(dzdr)\\
E_n(s,y)=&-N_n(s,y) \\
&+\int_0^{\lambda_n}\int_0^JS_{\mathbf{I}}(\lambda_n-r,y-z)
 g(v(s_n^-+r,z))\theta_{s_n^-}W(dzdr).
\end{align*}
More specifically,
\begin{itemize}
\item First, we take $V_n$ to be the first two terms, representing the 
contribution to $v(s,y)$ from the shifted initial conditions (both position 
and velocity).
\item Next, we realize the stochastic term as the sum of a conditionally 
Gaussian term and an error term.  The former is the stochastic term 
integrated over the light cone contained in the square
$\left\{(s,y)+D_n\right\}$, with the diffusion coefficient $g$ frozen at
$v(s_n^-,y) $. We call this term the noise term, $N_n$.
\item Finally, as mentioned above the error term $E_n$ is the difference 
between the stochastic term of $v(s,y)$ minus the noise term defined above.
\end{itemize}

As alluded to above, the noise term can be rewritten as:
\begin{equation} \label{eq:lower_noise_dist}
\begin{split}
N_n(s,y)&=g(v(s_n^-,y))\int_0^{\lambda_n}\int_{\left\{|z-y|
 \leq\lambda_n\right\}}S_{\mathbf{I}}\left(\lambda_n-r,y-z\right)
 \theta_{s_n^-}W(dzdr) \\
&=g(v(s_n^-,y))c_nZ
\end{split}
\end{equation}
where $c_n^2$ is the quadratic variation of the above double integral and 
$Z$ is a standard normal random variable. Moreover, for sufficiently small 
$\lambda_n$ relative to $J$, we have:
\begin{equation} 
\label{eq:lower_noise_var}
\begin{split}
c_n^2&=\int_0^{\lambda_n}\int_{\left\{|z-y|\leq\lambda_n\right\}}S_{\mathbf{I}}^2(r,y-z)dzdr\\
&=\frac{\lambda_n^2}{4}=2^{-2(1-2\epsilon)n-2}.
\end{split}
\end{equation}

\subsubsection{A Regularity Lemma}
Now, we find bounds for $E_n$ and $N_n$ by using H\"older continuity of 
$v$. Define the events
\begin{align*}
\mathbf{B}_n&=\left\{\sup_{(s,y)\in\Gamma_n(t,x)}\left|E_n(s,y)\right|
 \leq2^{-n}\right\}\\
\mathbf{C}_n&=\left\{\sup_{(s,y)\in\Gamma_n(t,x)}\left|N_n(s,y)\right|
		\leq2^{-(1-3\epsilon)n}\right\}\\
\mathbf{A}_n&=\mathbf{A}(K)\cap\mathbf{B}_n\cap\mathbf{C}_n.
\end{align*}
Then we assert the following:

\begin{lemma} 
\label{lem:lower_regularity}
$\sum_{n=1}^{\infty}\mathbf{P}\left(\mathbf{B}_n^c\right)<\infty$ and
$\sum_{n=1}^{\infty}\mathbf{P}\left(\mathbf{C}_n^c\right)<\infty$.  
\end{lemma}

To prove this lemma, we first establish a bound on the error term $E_n$. 
Recall its definition:
\begin{equation*}
\begin{split}
E_n&=-N_n(s,y)+\int_0^{\lambda_n}\int_0^JS_{\mathbf{I}}
 \left(\lambda_n-r,y-z\right)g(v(s_n^-+r,z))\theta_{s_n^-}W(dzdr)\\
&=-\int_0^{\lambda_n}\int_{\left\{|z-y|\leq\lambda_n\right\}}S_{\mathbf{I}}
 \left(\lambda_n-r,y-z\right)g(v(s_n^-,y))\theta_{s_n^-}W(dzdr)\\
&\quad+\int_0^{\lambda_n}\int_0^JS_{\mathbf{I}}\left(\lambda_n-r,y-z\right)
 g(v(s_n^-+r,z))\theta_{s_n^-}W(dzdr)
\end{split}
\end{equation*}
Note that in the integrals above, $S_{\mathbf{I}}(\lambda_n-r,y-z)=0$ 
outside of the light cone $|z-y|\leq\lambda_n$. Thus, we restrict the domain 
of integration of $z$:
\begin{equation*}
\begin{split}
E_n=&-\int_0^{\lambda_n}\int_{\left\{|z-y|\leq\lambda_n\right\}}
 S_{\mathbf{I}}\left(\lambda_n-r,y-z\right)g(v(s_n^-,y))
        \theta_{s_n^-} W(dz dr) \\
&+\int_0^{\lambda_n}\int_{\left\{|z-y|\leq\lambda_n\right\}}
 S_{\mathbf{I}}(\lambda_n-r,y-z)g(v(s_n^-+r,z))
 \theta_{s_n^-}W(dzdr)\\
=&\int_0^{\lambda_n}\int_{\left\{|z-y|\leq\lambda_n\right\}}
 S_{\mathbf{I}}\left(\lambda_n-r,y-z\right)\\
&\qquad\times\left[g(v(s_n^-+r,z))
 -g(v(s_n^-,y))\right]\theta_{s_n^-}W(dzdr)
\end{split}
\end{equation*}
	
We define the rectangle
\begin{equation*}
\Delta_n(s,y)=\left\{r\in\mathbb{R}_+,z\in[0,J]:|r-s|\leq\lambda_n,|z-y|
 \leq\lambda_n\right\}
\end{equation*}
and let $p$ be a positive integer. Then it follows that
\begin{multline*}
\mathbf{E}\left[E_n^{2p}\right]
=\mathbf{E}\bigg(\int_0^{\lambda_n}\int_{\left\{|z-y|\leq\lambda_n\right\}}
 S_{\mathbf{I}}\left(\lambda_n-r,y-z\right)\\
 \left[g(v(s_n^-+r,z))-g(v(s_n^-,y)
	)\right]\theta_{s_n^-}W(dzdr)\bigg)^{2p}
\end{multline*}
and since the integrand above is continuous in $\lambda_n$, we can use the
Burkholder-Davis-Gundy inequality to obtain:
\begin{equation*}
\begin{split}
\mathbf{E}\left[ E_n^{2p} \right]
&\lesssim_p\mathbf{E}\Bigg(\int_{s_n^-}^{s_n^-+\lambda_n}\int_{\left\{|z-y|\leq
 \lambda_n\right\}}S_{\mathbf{I}}^2\left(\lambda_n-\left(r-s_n^-\right),y-z\right)\\
&\quad\times\left[g(v(r,z))
 -g(v(s_n^-,y))\right]^2dzdr\Bigg)^{p}\\
&\lesssim_p\mathbf{E}\Bigg(\int_{s_n^-}^s\int_{\left\{|z-y|\leq\lambda_n\right\}}
 S_{\mathbf{I}}^2\left(s-r,y-z\right)\\
&\quad\times\left[g(v(r,z))
 -g(v(s_n^-,y))\right]^2dzdr\Bigg)^{p}.
\end{split}
\end{equation*}
As usual, the notation $a(x)\lesssim_{p}b(x)$ means that 
$a(x)\leq C_{p}b(x)$.  
	
Since $g$ is Lipschitz and since $(s_n^-,y)\in\Delta_n(s,y)$,
\begin{align*}
|g(v(r,z))&-g(v(s_n^-,y))|^2 \\
&\leq L_g^2\left|v(r,z)-v\left(s_n^-,y\right)\right|^2\\
&\leq L_g^2\left(2\left|v(r,z)-v(s,y)\right|^2
 +2\left|v(s,y)-v\left(s_n^-,y\right)\right|^2\right)\\
&\leq4L_g^2\sup_{(r,z)\in\Delta_n(s,y)}\left|v(r,z)-v(s,y)\right|^2
\end{align*}
	
With this bound, we get:
\begin{multline} 
\label{eq:lower_reg_power_bound}
\mathbf{E}\left[ E_n^{2p} \right]
 \lesssim\mathbf{E}\left[\sup_{(r,z)\in\Delta_n(s,y)}
 \left[\left|v(r,z)-v(s,y)\right|^{2p}\right]\right]\\
 \left(\int_{s_n^-}^s\int_{\left\{|z-y|\leq\lambda_n\right\}}
 S_{\mathbf{I}}^2\left(s-r,y-z\right)dzdr\right)^{p}.
\end{multline}
	
Recall that $v(s,y)$ is almost surely $\beta$-H\"older continuous for any 
$\beta<\frac{1}{2}$.
Setting $\beta=\frac{1}{2}-\frac{1}{2p}$, we obtain
\begin{align} 
\label{eq:modulus-v-moment}
\mathbf{E} &\left[ \sup_{(r,z) \in \Delta_n(s,y)}
 \left[\left|v(r,z)-v(s,y)\right|^{2p}\right]\right]  \\
&\qquad\lesssim_{g,p}\sup_{(r,z)\in\Delta_n(s,y)}\left(|r-s|^{\frac{1}{2}-\frac{1}{2p}}
 +|z-y|^{\frac{1}{2}-\frac{1}{2p}}\right)^{2p}\nonumber\\
&\qquad\lesssim_{g,p}\left(\lambda_n^{\frac{1}{2}-\frac{1}{2p}}\right)^{2p}
=\lambda_n^{p-1}=2^{-(1-2\epsilon)n(p-1)}.  \nonumber
\end{align}
Recalling \eqref{eq:lower_lambda}, we then bound the integral:
\begin{align} 
\label{eq:lower_reg_bound_2}
\bigg(\int_0^{\lambda_n}\int_{\left\{|z-y|\leq\lambda_n\right\}}&S_{\mathbf{I}}^2(r,y-z)dzdr\bigg)^p \nonumber\\
&=\left(\frac{1}{4}\int_0^{\lambda_n}\int_{\left\{|z-y|\leq\lambda_n\right\}}
\mathbf{1}_{\left\{|y-z|<r
 \right\}}dzdr\right)^p \\
&\lesssim \lambda_n^{2p} \nonumber\\
&\lesssim 2^{-2 (1 - 2 \epsilon) np}   \nonumber
\end{align}
so by \eqref{eq:lower_reg_power_bound}, \eqref{eq:modulus-v-moment}, and
\eqref{eq:lower_reg_bound_2}, we obtain a bound on the error term:

\begin{equation} 
\label{eq:lower_reg_bound}
\mathbf{E}\left[E_n^{2p}\right]\lesssim_{g,p}2^{-(1-2\epsilon)n(3p-1)}.
\end{equation}

\begin{proof}[Proof of Lemma \ref{lem:lower_regularity}]
Recalling that 
$\#\left\{\Gamma_n(t,x)\right\}\lesssim\lambda_n^{-2}=2^{(2-4\epsilon)n}$, 
we find:
\begin{align*}
\mathbf{P}\left(\mathbf{B}_n^c\right)
&=\mathbf{P}\left\{\sup_{(s,y)\in\Gamma_n(t,x)}
 \left|E_n(s,y)\right|>2^{-n}\right\}\\
&\leq\sum_{(s,y)\in\Gamma_n(t,x)}\mathbf{P}\left\{
 \left|E_n(s,y)\right|>2^{-n}\right\}\\
&\lesssim2^{(2-4\epsilon)n}\ \mathbf{P}\left\{
 \left|E_n(s,y)\right|>2^{-n}\right\}.
\end{align*}
By Markov's inequality, we can continue as follows,  
\begin{equation*}
\mathbf{P}\left(\mathbf{B}_n^c\right)
 \lesssim2^{(2-4\epsilon)n+2np}\ \mathbf{E}\left[E_n^{2p}\right]
\end{equation*}
after which we use \eqref{eq:lower_reg_bound} to obtain the bound:
\begin{equation*}	
\mathbf{P}\left(\mathbf{B}_n^c\right)\lesssim2^{(3-6\epsilon)n-(1-6\epsilon)np}.
\end{equation*}
Thus, the summation $\sum_{n=1}^{\infty}\mathbf{P}\left(\mathbf{B}_n^c\right)$ 
converges when $p>\frac{3-6\epsilon}{1-6\epsilon}$.

With a similar decomposition, we obtain:
\begin{equation*}
\mathbf{P}\left(\mathbf{C}_n^c\right)\lesssim2^{(4-4\epsilon)n}
 \ \mathbf{P}\left\{\left|N_n(s,y)\right|>2^{-(1-3\epsilon)n}\right\}.
\end{equation*}
Recalling \eqref{eq:lower_noise_dist} and \eqref{eq:lower_noise_var}, this 
implies:
\begin{equation*}
\begin{split}
\mathbf{P}\left(\mathbf{C}_n^c\right)&\lesssim2^{(4-4\epsilon)n}
 \ \mathbf{P}\left\{\left|g\left(v\left(s_n^-,y\right)\right)Z\right|
 >2^{(1-2\epsilon)n}2^{-(1-3\epsilon)n}\right\}\\
&\lesssim2^{(4-4\epsilon)n}\ \mathbf{P}\left\{|Z|>C_g^{-1}2^{\epsilon n}\right\}
\end{split}
\end{equation*}
where $Z$ is a standard normal random variable. Now, we use a standard tail 
estimate for the normal (often called the Chernoff bound) to conclude
\begin{equation*}
\mathbf{P}\left\{|Z|>C_g^{-1}\ 2^{\epsilon n}\right\}
 \leq2\exp{\left(-C_g^{-2}\ 2^{2\epsilon n-1}\right)}
\end{equation*}
and it follows that $\sum_{n=1}^{\infty}\mathbf{P}\left(\mathbf{C}_n^c\right)$ converges.
\end{proof}

\subsubsection{A Counting Lemma}

\begin{lemma} \label{lem:lower_counting}
For each $K>0$ there exists a constant $c_K$ such that for every 
$n\in\mathbf{N}$ and $(t,x)\in D_n$,
\begin{equation*}
\mathbf{E}\left[\#\left\{(s,y)\in\overline{\Gamma}_n(t,x):v(s,y)\leq2^{-n}K\right\};
 \mathbf{A}_n\right]\leq c_K.
\end{equation*}
\end{lemma}

The proof of Lemma \ref{lem:lower_counting} will require several preliminary steps.

We fix $(t,x)$ and order the points in $\Gamma_n(t,x)$ lexicographically, 
calling the $i$th point $(s_i,y_i)$ for some 
$i\in\mathcal{I}(t,x)=\{1,2,\dots,\#\{\Gamma_n(t,x)\}\}$ - i.e.,
if $i<j$ then $s_i\leq s_j$ and if $s_i=s_j$, then $x\leq x_i<x_j\mod J$. 
For given $(t,x)$, we define the set $\Delta^n(s,y)$ as follows:
\begin{equation*}
\Delta^n(s,y)=
\begin{cases}
 \left[0,s_n^-\right]\times\mathbf{I} &y=x\\
 \left(\left[0,s_n^-\right]\times\mathbf{I}\right)
\bigcup\left(\left[s_n^-,s\right]\times
 \left[x-\lambda_n,y-\lambda_n\right]
\right)&y\neq x
\end{cases}
\end{equation*}
where the interval $[x-\lambda_n,y-\lambda_n]$ on $\mathbf{I}$ is taken modulo $J$,
wrapping around whenever $x-\lambda_n>y-\lambda_n$.
(Note that this is not the same as the previously defined $\Delta_n(s,y)$).

Let $\mathcal{F}^n_i$ be the $\sigma$-algebra generated by $\dot{W}$ in the set
$\Delta^n\left(s_i,y_i\right)$. Then $V_n\left(s_i,y_i\right)$ is
$\mathcal{F}^n_i$-measurable for all $i\in\mathbb{N}$. Recall from 
(\ref{eq:lower_noise_dist}) that
\begin{equation} 
\label{eq:lower_noise_n_dist}
N_n\left(s_i,y_i\right)=g\left(v\left(\left(s_i\right)_n^-,y_i\right)\right)c_nZ_i
\end{equation}
where $c_n=2^{-(1-2\epsilon)n-1}$ and $Z_i\sim\mathcal{N}(0,1)$ is
$\mathcal{F}^n_{i+1}$-measurable but independent of $\mathcal{F}^n_i$.

Let $\mathbf{P}^n_i$ denote the conditional probability with respect to 
$\mathcal{F}^n_i$ and let $\delta>1$ be a constant depending only on $K$. 
Define
\begin{equation*}
\overline{v}_n(s,y)=V_n(s,y)+N_n(s,y).
\end{equation*}
We now prove the following lemma:

\begin{lemma} \label{lem:lower_initial}
There exists $ d_K > 0 $ such that for all $ i \in \mathcal{I}(t, x) $,
\begin{equation} \label{eq:lower_initial}
\mathbf{P}^n_i\Big[\overline{v}_n\left(s_i,y_i\right)\leq-2^{-n}\ \Big|\
 \overline{v}_n\left(s_i,y_i\right)\leq2^{-n}(K+1)\Big]\geq d_K
\end{equation}
almost surely on the event $\{V_n(s_i,y_i)\leq\delta2^{-(1-\epsilon)n}\}$.
\end{lemma}

\begin{proof}
From the definition of conditional probability, the left hand side of 
(\ref{eq:lower_initial}) is:
\begin{equation*}
\begin{split}
H&:=\mathbf{P}^n_i\Big[\overline{v}_n\left(s_i,y_i\right)\leq-2^{-n}
        \ \Big|\ \overline{v}_n\left( s_i, y_i \right) \leq 2^{-n} (K + 1) \Big] \\
&= \frac{\mathbf{P}^n_i\left\{ \overline{v}_n\left( s_i, y_i \right) \leq-2^{-n} \right\}}
        {\mathbf{P}^n_i\left\{ \overline{v}_n\left( s_i, y_i \right) \leq 2^{-n} (K + 1) \right\}} \\
&= \frac{\mathbf{P}^n_i\left\{ N_n\left( s_i, y_i \right) \leq -2^{-n}
                - V_n\left( s_i, y_i \right) \right\}}
        {\mathbf{P}^n_i\left\{ N_n\left( s_i, y_i \right) \leq 2^{-n}(K + 1)
                - V_n\left( s_i, y_i \right) \right\}}.
\end{split}
\end{equation*}
Using \eqref{eq:lower_noise_n_dist}, we find that
\begin{equation*}
H = \frac{\mathbf{P}^n_i\left\{ g\left( v\left( \left( s_i \right)_n^-, y_i \right) \right)
                c_n Z_i \leq -2^{-n} - V_n\left( s_i, y_i \right) \right\}}
                {\mathbf{P}^n_i\left\{ g\left( v\left( \left( s_i \right)_n^-, y_i \right) \right) c_n Z_i
                \leq 2^{-n}(K + 1) - V_n\left( s_i, y_i \right) \right\}}
\end{equation*}
and using \eqref{eq:lower_noise_var}, we find that
\begin{equation*}
H = \frac{\mathbf{P}^n_i\left\{ g\left( v\left( \left( s_i \right)_n^-, y_i \right) \right) Z_i
                \leq -2^{1-2 \epsilon n} - 2^{1+(1 - 2 \epsilon) n} V_n\left( s_i, y_i \right) \right\}}
                {\mathbf{P}^n_i\left\{ g\left( v\left( \left( s_i \right)_n^-, y_i \right) \right) Z_i
                \leq 2^{1-2 \epsilon n}(K + 1) - 2^{1+(1 - 2 \epsilon) n} V_n\left( s_i, y_i \right)\right\}}.
\end{equation*}
Then define $\rho_{n,i}=2g(v((s_i)_n^-,y_i))^{-1}$.
Note that $\rho_{n,i}$ is almost surely bounded above by $2C_g$ and 
below by $2c_g^{-1}>0 $, both uniformly in $n$ and $i$. Plugging this 
into the above equation, we find:
\begin{equation} 
\label{eq:lower_counting_ratio}
H = \frac{\mathbf{P}^n_i\left\{ Z_i \leq -\rho_{n, i}\left( 2^{-2 \epsilon n}
                + 2^{(1 - 2 \epsilon) n} V_n\left( s_i, y_i \right) \right) \right\}}
        {\mathbf{P}^n_i\left\{ Z_i \leq -\rho_{n, i}\left( -2^{-2 \epsilon n}(K + 1)
                + 2^{(1 - 2 \epsilon) n} V_n\left( s_i, y_i \right) \right) \right\}}.
\end{equation}

Now we examine $ H $ in two cases. The first case is on the event
\begin{equation} 
\label{eq:lower_counting_one}
\left\{-2^{-2\epsilon n}(K+1)+2^{(1-2\epsilon)n}V_n\left(s_i,y_i\right)
                \leq 0 \right\}.
\end{equation}
Since the denominator in (\ref{eq:lower_counting_ratio}) is less than or 
equal to $1$, we can bound $H$ below by its numerator:
\begin{equation*}
H\geq\mathbf{P}^n_i\left\{Z_i\leq-\rho_{n,i}\left(2^{-2\epsilon n}
 +2^{(1-2\epsilon)n}V_n\left(s_i,y_i\right)\right)\right\}.
\end{equation*}
Using the decomposition
\begin{equation*}
\begin{split}
&2^{-2\epsilon n}+2^{(1-2\epsilon)n}V_n\left(s_i,y_i\right)\\
&=\left(2^{-2\epsilon n}(K+2)\right)+\left(-2^{-2\epsilon n}(K+1)
 +2^{(1-2\epsilon)n}V_n\left(s_i,y_i\right)\right)\\
&\leq\left(2^{-2\epsilon n}(K+2)\right)
\end{split}
\end{equation*}
and the assumption in \eqref{eq:lower_counting_one}, we find
\begin{equation} 
\label{eq:lower_counting_one_bound}
H\geq\mathbf{P}^n_i\left\{Z_i\leq-\rho_{n,i}2^{-2\epsilon n}(K+2)\right\}.
\end{equation}
Since $\rho_{n,i}\leq2C_g$ for all $n$, we note that for all $K>0$,
$\rho_{n,i}2^{-2\epsilon n}(K+2)\rightarrow0$ as $n\rightarrow\infty$. So for
sufficiently large $n$ (depending on $K$), $H\geq1/3$. Hence in the case given
by (\ref{eq:lower_counting_one_bound}), Lemma \ref{lem:lower_initial} follows.
	
The second case is on the event
\begin{equation} \label{eq:lower_counting_two}
\left\{-2^{-2\epsilon n}(K+1)+2^{(1-2\epsilon)n}V_n\left(s_i,y_i\right)
              > 0 \right\}.
\end{equation}
Here, we use the following inequality from Lemma 8 of Mueller and Pardoux 
\cite{mp99}:  For $a,b>0$ and $Z$ a standard normal random variable,
\begin{equation*}
\frac{\mathbf{P}\{Z>a\}}{\mathbf{P}\{Z>a+b\}}\leq\frac{1}{2\mathbf{P}\{Z>1\}}
 \vee\left(1+\frac{\sqrt{e}}{1-e^{-1}}(a+b)be^{ab+\frac{b^2}{2}}\right).
\end{equation*}
Let $a=\rho_{n,i}(-2^{-2\epsilon n}(K+1)+2^{(1-2\epsilon)n}V_n(s_i,y_i))$ and 
$b=\rho_{n,i}2^{-2\epsilon n}(K+2)$.
Recalling that from our given conditions, 
$V_n(s_i,y_i)\leq\delta2^{-(1-\epsilon)n}$ almost surely, we find that:
\begin{equation*}
\begin{split}
(a+b)b&=\rho_{n,i}^2\left(2^{-2\epsilon n}+2^{(1-2\epsilon)n}V_n\left(
 s_i,y_i\right)\right)(K+2)2^{-2\epsilon n}\\
&\leq\rho_{n,i}^2\left(2^{-2\epsilon n}+2^{-2\epsilon n}\delta\right)(K+2)
        2^{-2 \epsilon n} \\
&=\rho_{n,i}^2(\delta+1)(K+2)2^{-4\epsilon n}\\
&\leq\rho_{n,i}^2(\delta+1)(K+2)
\end{split}
\end{equation*}
and
\begin{equation*}
\begin{split}
\left(a+\frac{b}{2}\right)b&=\rho_{n,i}^2\left(-(0.5K)2^{-2\epsilon n}
 +2^{(1-2\epsilon)n}V_n\left(s_i,y_i\right)\right)(K+2)2^{-2\epsilon n}\\
&\leq\rho_{n,i}^2\left(-(0.5K)2^{-2\epsilon n}+2^{-2\epsilon n}\delta\right)
 (K+2) 2^{-2 \epsilon n} \\
&= \rho_{n, i}^2 (\delta - 0.5 K)(K + 2) 2^{-4 \epsilon n} \\
&\leq\rho_{n,i}^2(\delta-0.5)(K+2)\qquad\qquad\text{(recalling that $K>1$)}
\end{split}
\end{equation*}
almost surely. Using these results with (\ref{eq:lower_counting_ratio}), we find:
\begin{equation*}
\begin{split}
H &= \frac{\mathbf{P}\{Z \leq -(a + b)\}}{\mathbf{P}\{Z \leq -a\}} \\
&= \frac{\mathbf{P}\{Z > a + b\}}{\mathbf{P}\{Z > a\}} \\
&\geq \left( 2 \mathbf{P}\{Z > 1\} \right) \wedge
 \left(1+\frac{\sqrt{e}}{1-e^{-1}}(a+b)be^{ab+\frac{b^2}{2}}\right)^{-1}\\
&\geq \left( 2 \mathbf{P}\{Z > 1\} \right) \wedge
 \left(1+\frac{\sqrt{e}}{1-e^{-1}}\rho_{n,i}^2(\delta+1)(K+2)
 e^{\rho_{n,i}^2(\delta-0.5)(K+2)}\right)^{-1}.
\end{split}
\end{equation*}
Since $\rho_{n,i}$ is almost surely uniformly bounded away from $0$ in $n$ 
and $i$, there exists $c_{g,\delta}>0 $ such that 
$\rho_{n,i}^2\geq 4c_{g,\delta}$. So:
\begin{multline*}
\left(1+\frac{\sqrt{e}}{1-e^{-1}}\rho_{n,i}^2(\delta+1)(K+2)
                e^{\rho_{n, i}^2 (\delta - 0.5)(K + 2)} \right)^{-1} \\
\geq c_{g,\delta}\left(c_{g,\delta}^{-1}+\frac{\sqrt{e}}{1-e^{-1}}(\delta+1)(K+2)
 e^{c_{g,\delta}(\delta-0.5)(K+2)}\right)^{-1}
\end{multline*}
and since $\delta$ depends only on $K$, the right hand side above is bounded 
below by some $\gamma_{K,g}>0 $. Then $H$ is bounded above by
\begin{equation*}
        d_K = 2 \mathbf{P}\{Z > 1\} \wedge \gamma_{K,g} > 0
\end{equation*}
in the case given by \eqref{eq:lower_counting_two} as well. Hence Lemma 
\ref{lem:lower_initial} follows in both cases.
\end{proof}

\begin{proof}[Proof of Lemma \ref{lem:lower_counting}]
Define
\begin{equation*}
\xi_n=\mathbf{E}\Big[\#\Big\{(s,y)\in\overline{\Gamma}_n(t,x):
                v(s, y) \leq 2^{-n} K \Big\}; \mathbf{A}_n \Big].
\end{equation*}
From the definitions of $\mathbf{A}_n$ and $\overline{\Gamma}$, it follows 
that $\xi_n$ is bounded by:
\begin{multline*}
\xi_n\leq\mathbf{E}\Big[\#\Big\{(s,y)\in\overline{\Gamma}_n(t,x):
 0<v(s,y)\leq2^{-n}K,\\
 \left|E_n(s,y)\right|\leq2^{-n},\left|N_n(s,y)\right|\leq2^{-2(1-3\epsilon)n}\Big\}\Big].
\end{multline*}
Recall that by definition, $\overline{v}_n=V_n(s,y)+N_n(s,y)=v(s,y)-E_n(s,y)$. Thus, we obtain the bound
\begin{multline*}
\xi_n\leq\mathbf{E}\Big[\#\Big\{(s,y)\in\overline{\Gamma}_n(t,x):
 -2^{-n}<\overline{v}_n(s,y)\leq2^{-n}(K+1),\\
 \left|E_n(s,y)\right|\leq2^{-n},\left|N_n(s,y)\right|\leq2^{-2(1-3\epsilon)n}\Big\}\Big].
\end{multline*}
Note that if $V_n(s,y)+N_n(s,y)\leq2^{-n}(K+1)$ and
$|N_n(s,y)|\leq2^{-2(1-3\epsilon)n}$, then for some $\delta>1$
depending on $K,\epsilon$ we have $V_n(s,y)\leq\delta2^{-(1-\epsilon)n}$. 
So we can write:
\begin{multline*}
\xi_n\leq\mathbf{E}\Big[\#\Big\{(s,y)\in\overline{\Gamma}_n(t,x):
                -2^{-n} < \overline{v}_n(s, y) \leq 2^{-n}(K + 1), \\
                V_n(s, y) \leq \delta 2^{-(1 - \epsilon) n} \Big\} \Big].
\end{multline*}
	
Let $\{\sigma_n(k)\}_{k\in\mathbb{N}}$ be the sequence of indices
$i\in\mathcal{I}$, in lexicographical order, such that both
$\overline{v}_n(s_i,y_i)\leq2^{-n}(K+1)$ and
$V_n(s_i,y_i)\leq\delta2^{-(1-\epsilon)n}$.
	
Out of the set of points on $\Gamma_n$ such that 
$\overline{v}_n\leq2^{-n}(K+1)$, one looks at the points where
$\overline{v}_n<-2^{-n}$, which would force $v$ to be negative. Thus we 
define the event
\begin{equation*}
\mathbf{D}_k=\left\{
 \overline{v}_n\left(s_{\sigma_n(k)},y_{\sigma_n(k)}\right)\leq-2^{-n}\right\}
\end{equation*}
and for $k\in\mathbb{N}$, we define the indicator random variable 
\begin{equation*}
I_k=1_{\mathbf{D}_k}.
\end{equation*}

From Lemma \ref{lem:lower_initial}, it is clear that
\begin{equation*}
\mathbf{P}\left\{I_1=1\right\}\geq d_K
\end{equation*}
and moreover, since $V_i$ and $N_i$ are $\mathcal{F}^n_j$-measurable for all 
$i<j$, we can also use Lemma \ref{lem:lower_initial} to find that
\begin{equation*}
\mathbf{P}\Big[I_k=1\ \Big|\ I_1,\dots,I_{k-1}\Big]\geq d_K
\end{equation*}
for $k>1$. Finally, let
\begin{equation*}
\overline{\sigma}_n=\inf\left\{k;I_k=1\right\}.
\end{equation*}
Since $v(s,y)\geq0$ for all $(s,y)\in\overline{\Gamma}_n$, it follows that
$\xi_n\leq\mathbf{E}\overline{\sigma}_n$.

Note that the $I_k$'s are not independent. We couple $\left\{I_k\right\}$ 
with a sequence of independent random variables $\left\{Y_k\right\}$ as 
follows: Let $\left\{U_k\right\}_{k\geq1}$ be a sequence of mutually 
independent random variables that are globally independent of the $I_k$'s 
and have uniform law on $[0,1]$. Then define
\begin{equation*}
Y_k =
\begin{cases}
0 & \text{if } I_k = 0 \text{ or } U_k > d_K/
        \mathbf{P}\Big[ I_k = 1 \ \Big|\ I_1, \dots, I_{k-1} \Big] \\
1 & \text{if } I_k = 1 \text{ and } U_k \leq d_K/
        \mathbf{P}\Big[ I_k = 1 \ \Big|\ I_1, \dots, I_{k-1} \Big]
\end{cases}
\end{equation*}

for $k\geq1$. Then clearly,
\begin{equation*}
Y_k\leq I_k
\end{equation*}
and for $ k > 1 $,
\begin{equation*}
\begin{split}
\mathbf{P}\Big[Y_k=1\ \Big|\ Y_1,\dots,Y_{k-1}\Big]
&=\mathbf{P}\Big[Y_k=1\ \Big|\ I_1,\dots,I_{k-1}\Big]\\
&=d_K
\end{split}
\end{equation*}
so $\left\{ Y_k \right\} $ is a sequence of i.i.d. random variables. Let
$ \tilde{\sigma} = \inf\left\{ k; Y_k = 1 \right\} $. Then
\begin{align*}
\overline{\sigma}_n &= \text{1st } k \text{ such that } I_k = 1 \\
\tilde{\sigma} &= \text{1st } k \text{ such that } Y_k = 1
\end{align*}
and since $ Y_k \leq I_k $, it follows that $ \overline{\sigma}_n \leq \tilde{\sigma} $. So 
\begin{equation*}
\mathbf{E}\overline{\sigma}_n \leq \mathbf{E}\tilde{\sigma} = d_K^{-1}.
\end{equation*}
\end{proof}

\subsection{Lemma \ref{lem:lower_main}, Conclusion}

Finally, we cite a measure-theoretic result related to the Borel-Cantelli Lemma:

\begin{lemma} \label{lem:borel}
Let $\left\{X_n\right\}$ be a sequence of $[0,\infty)$-valued random 
variables, and $\left\{\mathbf{F}_n\right\}$ be a sequence of events, 
such that both:
\begin{gather*}
\sum_{n=0}^{\infty} \mathbf{P}\left( \mathbf{F}_n^c \right) < \infty \\
\sum_{n=0}^{\infty} \mathbf{E}\left[ X_n; \mathbf{F}_n \right] < \infty
\end{gather*}
Then $\sum_{n=0}^{\infty}X_n<\infty$ almost surely.
\end{lemma}

\begin{proof}
Let $\mathbf{F}=\{\sum_{n=0}^{\infty}X_n=+\infty\}$. Then on the event
$\mathbf{F}\cap\left(\liminf\mathbf{F}_n\right)$, we have
$\sum_{n=0}^{\infty}X_n\mathbf{1}_{\mathbf{F}_n}=+\infty$. So from the 
second condition, we get $\mathbf{P}(\mathbf{F}\cap\liminf\mathbf{F}_n)=0$. 
However, from the first condition and Borel-Cantelli, we find:
\begin{equation*}
\mathbf{P}\left( \limsup \mathbf{F}_n^c \right) = 0
        \Rightarrow \mathbf{P}\left( \liminf \mathbf{F}_n \right) = 1
\end{equation*}
Implying that $\mathbf{P}(\mathbf{F})=0 $, which is our desired result.
\end{proof}

\begin{proof}[Proof of Lemma \ref{lem:lower_main}]
From equations \eqref{eq:lower_lebesgue_bound} and 
\eqref{eq:lower_lebesgue_count}, we have:
\begin{equation*}
\begin{split}
& 1_{A(K)} \int_0^{\tau^{(v)} \wedge T} \int_\mathbf{I} v(t, x)^{-2 \alpha} \ dx dt \\
&\leq 1_{A(K)} \sum_{n=0}^{\infty} \Big[ 2^{2 \alpha (n + 1)} K^{-2 \alpha} \\
        &\qquad \times\mu \left( \left\{ (t, x) \in \left[0, \tau^{(v)} \wedge T \right]
  \times \mathbf{I}:2^{-n - 1}K<v(t,x)\leq2^{-n}K\right\}\right)\Big].
\end{split}
\end{equation*}
Now consider the above expression $\mu(\cdots)$.  First, we can bound this 
expression above by dropping the inequality $2^{-n-1}K<v(t,x)$, which 
enlarges the set under consideration.  Secondly, we note that 
\begin{equation*}
\bigcup_{(t,x)\in D_n}\overline{\Gamma}_n(t,x)
\end{equation*}
covers the entire $(t,x)$-plane, because $\overline{\Gamma}_n(t,x)$ consists 
of the corners of rectangles which are translations of $D_n$, further 
translated by $(t,x)$.  So we can continue as follows.  
\begin{equation*}
\begin{split}
1_{A(K)}& \int_0^{\tau^{(v)} \wedge T} \int_\mathbf{I} v(t, x)^{-2 \alpha} \ dx dt \\
&\leq 1_{A(K)} \sum_{n=0}^{\infty} \Big[ 2^{2 \alpha (n + 1)} K^{-2 \alpha} \\
&\qquad    \times\iint_{D_n} \# \left\{ (s, y) \in \overline{\Gamma}_n(t, x) : v(s, y)
        \leq 2^{-n} K \right\} \ dx dt \Big] \\
&\quad + 1_{A(K)} \sum_{n=0}^{\infty} \Big[ 2^{2 \alpha (n + 1)} K^{-2 \alpha}
        \ \mu\left( J_n \right) \Big].
\end{split}
\end{equation*}
Now consider the summation of expectations
\begin{equation*}
\begin{split}
&\sum_{n=0}^{\infty}\mathbf{E}\bigg[1_{A(K)}\ 2^{2\alpha(n+1)}K^{-2\alpha}\\
&\qquad\iint_{D_n}\#\left\{(s,y)\in\overline{\Gamma}_n(t,x):v(s,y)
        \leq 2^{-n} K \right\} \ dx dt ; B_n \cap C_n \bigg] \\
&=\sum_{n=0}^{\infty}\mathbf{E}\bigg[2^{2\alpha(n+1)}K^{-2\alpha}\\
        &\qquad \iint_{D_n} \# \left\{ (s, y) \in \overline{\Gamma}_n(t, x) : v(s, y)
        \leq 2^{-n} K \right\} \ dx dt ; A_n \bigg]
\end{split}
\end{equation*}
Recalling that $\alpha<1$, and that $D_n$ was defined in 
\eqref{eq:lower_square}, we note that:
\begin{equation*}
\iint_{D_{n}} \ dx dt = 2^{-(2 - 4 \epsilon) n + 1}
\end{equation*}
so using Lemma \ref{lem:lower_counting}, we find:
\begin{equation*}
\begin{split}
& \sum_{n=0}^{\infty} \mathbf{E} \bigg[ 2^{2 \alpha (n + 1)} K^{-2 \alpha} \\
& \qquad  \times\iint_{D_n} \# \left\{ (s, y) \in \overline{\Gamma}_n(t, x) : v(s, y)
        \leq 2^{-n} K \right\} \ dx dt ; A_n \bigg] \\
&= \sum_{n=0}^{\infty} 2^{2 \alpha (n + 1)} K^{-2 \alpha}   \\
&\qquad \times\iint_{D_n} \mathbf{E} \left[
        \# \left\{ (s, y) \in \overline{\Gamma}_n(t, x) : v(s, y) \leq 2^{-n} K \right\} ; A_n
        \right] \ dx dt \\
&\leq \sum_{n=0}^{\infty} 2^{2 \alpha (n + 1)} K^{-2 \alpha} c_K \ \iint_{D_{n}} \ dx dt \\
&= \sum_{n=0}^{\infty} c_K K^{-2 \alpha} 2^{2 \alpha + 1} 2^{(2 \alpha - 2 + 4\epsilon) n}
\end{split}
\end{equation*}
and since $\alpha<1-2\epsilon$, the summation converges. Thus from using Lemma
\ref{lem:lower_regularity}, Lemma \ref{lem:borel}, and 
\eqref{eq:lower_strip_bound}, we have
\begin{equation} 
\label{eq:lower_main_k}
1_{A(K)}\int_0^{\tau^{(v)}\wedge T}\int_0^Jv(t,x)^{-2\alpha}dxdt<\infty
\end{equation}
almost surely. Observe that (\ref{eq:lower_k}) and (\ref{eq:lower_main_k}) 
imply (\ref{eq:lower_main}) since
\begin{equation*}
\begin{split}
&\int_0^{\tau^{(v)}\wedge T}\int_\mathbf{I}v(t,x)^{-2\alpha}dxdt\\
&=1_{A(K)}\int_0^{\tau^{(v)}\wedge T}\int_\mathbf{I}v(t,x)^{-2\alpha}dxdt
 +1_{A(K)^c}\int_0^{\tau^{(v)}\wedge T}\int_\mathbf{I}v(t,x)^{-2\alpha}dxdt\\
&\leq1_{A(K)}\int_0^{\tau^{(v)}\wedge T}\int_\mathbf{I}v(t,x)^{-2\alpha}dxdt
 +\int_0^{\tau^{(v)}\wedge T}\int_\mathbf{I}K^{-2\alpha}dxdt\\
&<\infty
\end{split}
\end{equation*}
almost surely.  Indeed, on $A(K)^c$ one knows that $v(t,x)>K$ for 
$(t,x)\in[0,T]\times[0,J]$, so $v^{-2\alpha}(t,x)\leq K^{-2\alpha}$ in this 
situation.  
\end{proof}

\appendix
\section{Proof of Theorem \ref{th:holder}}

\begin{proof}
For the remainder of this section we will simply write $N(t,x)$ instead of 
$N_\rho(t,x)$.  
Fix $T>0$ and consider the space and time differences, given respectively by
$|N(t,x+k)-N(t,x)|$ and $|N(t+h,x)-N(t,x)|$.  Without loss of generality, 
let $h,k\in[0,\frac{J}{2}]$.

\subsubsection{Space Difference}
To bound $N(t,x+k)-N(t,x)$ we first write
\begin{multline*}
\Delta_k^p=\mathbf{E}\big[\left|N(t,x+k)-N(t,x)\right|^p\big] \\
=\mathbf{E}\left[\left|\int_0^t\int_0^J\rho(s,y)
 \Big(S_{\mathbf{I}}(t-s,x+k-y)-S_{\mathbf{I}}(t-s,x-y)\Big)W(dyds)\right|^p\right]
\end{multline*}
and note that for all $r$, the following stochastic integral is a martingale 
over $t$:
\begin{equation*}
\int_0^t\int_0^J\rho(s,y)\Big(S_{\mathbf{I}}(r-s,x+k-y)-S_{\mathbf{I}}(r-s,x-y)\Big)W(dy ds).
\end{equation*}
We fix $p>2 $ and use Burkholder's inequality to find a constant $C_p$ 
such that
\begin{multline*}
\mathbf{E}\left[\left|\int_0^t\int_0^J\rho(s,y)
 \left(S_{\mathbf{I}}(r-s,x+k-y)-S_{\mathbf{I}}(r-s,x-y)\right)W(dyds)\right|^p\right]\\
\leq C_p\mathbf{E}\left[\left|\int_0^t\int_S\left|\rho(s,y)\right|^2
 \left|S_{\mathbf{I}}(r-s,x+k-y)-S_{\mathbf{I}}(r-s,x-y)\right|^2dyds\right|^{p/2}\right]
\end{multline*}
for all $r$. As $C_p$ does not depend on $r$, we can set $r=t$ to obtain
\begin{align*}
&\Delta_k^p\leq \\
&\; C_p\mathbf{E}\left[\left|\int_0^t\int_0^J\left|\rho(s,y)\right|^2
 \left|S_{\mathbf{I}}(t-s,x+k-y)-S_{\mathbf{I}}(t-s,x-y)\right|^2dyds\right|^{p/2}\right].
\end{align*}
Now we use H\"older's Inequality with exponents $\frac{p}{2}$ and $\frac{p}{p-2}$:
\begin{multline*}
\Delta_k^p\leq C_p\left(\mathbf{E}\int_0^t\int_0^J\rho(v(s,y)^pdyds\right) \\ 
\times\left[\int_0^t\int_0^J\left|S_{\mathbf{I}}(t-s,x+k-y)
 -S_{\mathbf{I}}(t-s,x-y)\right|^{2p/(p-2)}dyds\right]^{p/2-1}.
\end{multline*}
Since $\rho$ is almost surely bounded, the expectation
$\mathbf{E}\int_0^t\int_S\rho(s,y)^pdyds$
is bounded by a constant depending on
$T$ and $p$. So we obtain:
\begin{equation*}
\begin{split}
\Delta_k^p&\lesssim_{p,T}\left[\int_0^t\int_0^J
 \left|S_{\mathbf{I}}(t-s,x+k-y)-S_{\mathbf{I}}(t-s,x-y)
		\right|^{2p/(p-2)}dyds\right]^{p/2-1} \\
&=\left[\int_0^t\int_0^J\left|S_{\mathbf{I}}(s,x+k-y)
 -S_{\mathbf{I}}(s,x-y)\right|^{2p/(p-2)}dyds\right]^{p/2-1} \\
&=\left[\int_0^t\int_0^J\left|\sum_{m\in\mathbb{Z}}\frac{1}{2}
 1_{\left\{|x+k-y+mJ|<s\right\}}-\sum_{m\in\mathbb{Z}}\frac{1}{2}
 1_{\left\{|x-y+mJ|<s\right\}}\right|^{2p/(p-2)}dyds\right]^{p/2-1}\\
&=\frac{1}{2^p}\left[\int_0^t\int_0^J\left|\sum_{m\in\mathbb{Z}}\left(
 1_{\left\{|x+k-y+mJ|<s\right\}}-1_{\left\{|x-y+mJ|<s\right\}}\right)
		\right|^{2p/(p-2)}dyds\right]^{p/2-1}.
\end{split}
\end{equation*}
Here and in the next few estimates, the infinite sum over $m\in\mathbb{Z}$ 
is really finite, because the indicator functions will be zero for all but
finitely many values of $m$.  

We use the inequality $\left(\sum_{n=1}^N a_n\right)^p\lesssim_{p,N}\sum_{n=1}^N a_n^p$
to get:
\begin{equation*}
\begin{split}
\Delta_k^p&\lesssim_{p,T}\frac{1}{2^p}\left[\int_0^t\int_0^J\sum_{m\in\mathbb{Z}}\left|
 1_{\left\{|x+k-y+mJ|<s\right\}}-1_{\left\{|x-y+mJ|<s\right\}}
  \right|^{2p/(p-2)}dyds\right]^{p/2-1} \\
&=\frac{1}{2^p}\left[\int_0^t\int_0^J\sum_{m\in\mathbb{Z}}\left|1_{\left\{|x+k-y+mJ|
 <s\right\}}-1_{\left\{|x-y+mJ|<s\right\}}\right|dyds\right]^{p/2-1} \\
&=\frac{1}{2^p}\left[\int_0^t\int_{\mathbb{R}}\left|1_{\left\{|x+k-y|<s\right\}}
 -1_{\left\{|x-y|<s\right\}}\right|dyds\right]^{p/2-1}\\
&\leq\frac{1}{2^p}\left(2tk\right)^{p/2-1}\lesssim_{p,T}k^{p/2-1}.
\end{split}
\end{equation*}

\subsubsection{Time Difference}
As in the last section, we can use Burkholder's inequality to find:
\begin{multline*}
\mathbf{E}\big[|N(t+h,x)-N(t,x)|^p\big] \\
\leq C_p\mathbf{E}\left[\left|\int_0^t\int_0^J\left|\rho(s,y)\right|^2
\left|S_{\mathbf{I}}(t+h-s,x-y)-S_{\mathbf{I}}(t-s,x-y)\right|^2dyds\right|^{p/2}\right] \\
+C_p\mathbf{E}\left[\left|\int_t^{t+h}\int_0^J\left|\rho(s,y)\right|^2
 \left|S_{\mathbf{I}}(t+h-s,x-y)\right|^2dyds\right|^{p/2}\right].
\end{multline*}

The first term is handled like the space difference:
\begin{equation*}
\begin{split}
&\mathbf{E}\left[\left|\int_0^t\int_0^J\left|\rho(s,y)\right|^2
 \left|S_{\mathbf{I}}(t+h-s,x-y)-S_{\mathbf{I}}(t-s,x-y)\right|^2dyds\right|^{p/2}\right]\\	
&\lesssim_{p}\left(\mathbf{E}\int_0^t\int_0^J\rho(s,y)^pdyds\right)\\
&\quad\times\left[\int_0^t\int_0^J\left|S_{\mathbf{I}}(t+h-s,x-y)-S_{\mathbf{I}}(t-s,x-y)
	\right|^{2p/(p-2)}dyds\right]^{p/2-1}.
\end{split}
\end{equation*}
Using the inequality 
$\left(\sum_{n=1}^Na_n\right)^p\lesssim_{p,N}\sum_{n=1}^Na_n^p$, we can 
continue the above inequality
\begin{equation*}
\begin{split}
&\lesssim_{p,T}\left[\int_0^t\int_0^J
 \left|S_{\mathbf{I}}(s+h,x-y)-S_{\mathbf{I}}(s,x-y)\right|^{2p/(p-2)}
		dyds\right]^{p/2-1}\\
&=\left[\int_0^t\int_0^J\left|\sum_{m\in\mathbb{Z}}\frac{1}{2}
 1_{\left\{|x-y+mJ|<s+h\right\}}-\sum_{m\in\mathbb{Z}}\frac{1}{2}
 1_{\left\{|x-y+mJ|<s\right\}}\right|^{2p/(p-2)}dyds\right]^{p/2-1}\\
&\lesssim_p\frac{1}{2^p}\left[\int_0^t\int_{\mathbf{R}}\left|
 1_{\left\{|x-y|<s+h\right\}}-1_{\left\{|x-y|<s\right\}}\right|dyds
	\right]^{p/2-1}\\
&\leq\frac{1}{2^p}\left(2th\right)^{p/2-1}\lesssim_{p,T}h^{p/2-1}.
\end{split}
\end{equation*}

For the second term, we start with H\"older's inequality again:
\begin{equation*}
\begin{split}
&\mathbf{E}\left[\left|\int_t^{t+h}\int_0^J\left|\rho(s,y)\right|^2
	\left|S_{\mathbf{I}}(t+h-s,x-y)\right|^2dyds\right|^{p/2}\right]\\
&\leq C_p\left(\mathbf{E}\int_t^{t+h}\int_0^J\rho(s,y)^pdyds\right)\\
&\quad\times\left[\int_t^{t+h}\int_0^J\left|S_{\mathbf{I}}(t+h-s,x-y)\right|^{2p/(p-2)}dyds\right]^{p/2-1}.
\end{split}
\end{equation*}
Using the inequality 
$\left(\sum_{n=1}^Na_n\right)^p\lesssim_{p,N}\sum_{n=1}^Na_n^p$ one last 
time, we continue the inequality:
\begin{equation*}
\begin{split}
&\lesssim_{p,T}\left[\int_t^{t+h}\int_0^J\left|G(t+h-s,x-y)\right|
 ^{2p/(p-2)}dyds\right]^{p/2-1}\\
&=\left[\int_0^h\int_0^J\left|S_{\mathbf{I}}(s,x-y)\right|^{2p/(p-2)}dyds
		\right]^{p/2-1}\\
&=\left[\int_0^h\int_0^J\left|\frac{1}{2}1_{\left\{|x-y|<s\right\}}\right|
 ^{2p/(p-2)}dyds\right]^{p/2-1}\\
&=\frac{1}{2p}\left(h^2\right)^{p/2-1}\lesssim_{p}h^{p/2-1}.
\end{split}
\end{equation*}

\subsubsection{Conclusion}
Putting together the space and time differences, we obtain for $ h, k $:
\begin{equation*}
\begin{split}
\mathbf{E}\left[\left|N(t+h,x+k)-N(t,x)\right|^p\right]
&\lesssim_{p,T}h^{p/2-1}+k^{p/2-1}\\
&\lesssim_{p,T}\left(\sqrt{h^2+k^2}\right)^{p/2-1}.
\end{split}
\end{equation*}

Finally, we recall Kolmogorov's continuity theorem for multiparameter 
processes \cite{wal86}, Corollary 1.2.
\begin{theorem}[Kolmogorov]
Let $R$ be a rectangle in $\mathbb{R}^n$ and $\left\{X_t,t\in R\right\}$ be 
a real-valued stochastic process.  Suppose there exist $a,b,K$ all 
positive such that for all $s,t\in R$
\begin{equation*}
\mathbf{E}\left|X_t-X_s\right|^a\leq K|t-s|^{n+b}.
\end{equation*}
Then,
\begin{enumerate}
\item $X$ has a continuous realization;
\item there exist a constant $C$ depending only on $a,b,n$ and a random 
variable $Y$ such that with probability one,
\begin{equation*}
\left|X_t-X_s\right|\leq Y|t-s|^{b/a}
 \qquad\left(\forall\right)s,t\in R
\end{equation*}
and $\mathbf{E}\left[Y^a\right]\leq CK$;
\item if $\mathbf{E}\left[\left|X_t\right|^a\right]<\infty$ for some $t\in R$ then
\begin{equation*}
\mathbf{E}\left[\sup_{t\in R}\left|X_t\right|^a\right]<\infty.
\end{equation*}
\end{enumerate}
\end{theorem}

Setting $a=p$ and $b=p/2-3$, we obtain
\begin{equation*}
\begin{split}
\left|N(t+h,x+k)-N(t,x)\right|&\leq Y\left(\sqrt{h^2+k^2}\right)^{b/a}\\
&\leq Y\left(h^{1/2-3/p}+k^{1/2-3/p}\right)
\end{split}
\end{equation*}
where $\mathbf{E}[Y]<\infty $ and depends only on $p$ and $T$. Then since 
$\beta<1/2$, we can set $p=\frac{6}{1-2\beta} $ and the conclusion follows.
\end{proof}

%\bibliography{bibtex}

\begin{thebibliography}{DKN07}

\bibitem[DKN07]{DKN07}
R.~C. Dalang, D.~Khoshnevisan, and E.~Nualart, \emph{Hitting probabilities for
  systems of non-linear stochastic heat equations with additive noise}, ALEA
  Lat. Am. J. Probab. Math. Stat. \textbf{3} (2007), 231--271. \MR{2365643}

\bibitem[DPZ92]{dz92}
G.~Da~Prato and J.~Zabczyk, \emph{Stochastic equations in infinite dimensions},
  Encyclopedia of Mathematics and its Applications, vol.~44, Cambridge
  University Press, Cambridge, 1992.

\bibitem[DSS10]{DS10}
R.~C. Dalang and M.~Sanz-Sol\'e, \emph{Criteria for hitting probabilities with
  applications to systems of stochastic wave equations}, Bernoulli \textbf{16}
  (2010), no.~4, 1343--1368. \MR{2759182}

\bibitem[DSS15]{DS15}
\bysame, \emph{Hitting probabilities for nonlinear systems of stochastic
  waves}, Mem. Amer. Math. Soc. \textbf{237} (2015), no.~1120, v+75.
  \MR{3401290}

\bibitem[Eva98]{eva98}
Lawrence~C. Evans, \emph{Partial differential equations}, Graduate Studies in
  Mathematics, vol.~19, American Mathematical Society, Providence, RI, 1998.

\bibitem[MP99]{mp99}
C.~Mueller and E.~Pardoux, \emph{The critical exponent for a stochastic
  {P}{D}{E} to hit zero}, Stochastic analysis, control, optimization and
  applications, Birkh\"auser Boston, Boston, MA, 1999, pp.~325--338.

\bibitem[MP10]{mper10}
P.~M{\"o}rters and Y.~Peres, \emph{Brownian motion}, Cambridge Series in
  Statistical and Probabilistic Mathematics, Cambridge University Press,
  Cambridge, 2010.

\bibitem[MT02]{mt01}
C.~Mueller and R.~Tribe, \emph{Hitting properties of a random string},
  Electronic J. Prob. \textbf{7} (2002), 1--29, Paper no. 10.

\bibitem[Mue98]{mue98}
C.~Mueller, \emph{Long-time existence for signed solutions to the heat equation
  with a noise term}, Probab. Theory Relat. Fields \textbf{110}, (1998), no.~1, 51--68.

\bibitem[NP94]{np94}
D.~Nualart and E.~Pardoux, \emph{Markov field properties of solutions of white
  noise driven quasi-linear parabolic {P}{D}{E}s}, Stochastics Stochastics Rep.
  \textbf{48} (1994), no.~1-2, 17--44.

\bibitem[NV09]{NV09}
E.~Nualart and F.~Viens, \emph{The fractional stochastic heat equation on the
  circle: Time regularity and potential theory}, Stochastic Processes and their
  Applications \textbf{119} (2009), no.~5, 1505--1540.

\bibitem[RY99]{ry99}
Daniel Revuz and Marc Yor, \emph{Continuous martingales and {B}rownian motion},
  third ed., Grundlehren der Mathematischen Wissenschaften [Fundamental
  Principles of Mathematical Sciences], vol. 293, Springer-Verlag, Berlin,
  1999.

\bibitem[Wal86]{wal86}
J.B. Walsh, \emph{An introduction to stochastic partial differential
  equations}, \'Ecole d'\'et\'e de probabilit\'es de Saint-Flour, XIV-1984
  (Berlin, Heidelberg, New York) (P.~L. Hennequin, ed.), Lecture Notes in
  Mathematics 1180, Springer-Verlag, 1986, pp.~265--439.

\end{thebibliography}
%\bibliographystyle{amsalpha}

\def\cprime{$'$} \def\cprime{$'$} \def\cprime{$'$}
\providecommand{\bysame}{\leavevmode\hbox to3em{\hrulefill}\thinspace}
\providecommand{\MR}{\relax\ifhmode\unskip\space\fi MR }
\providecommand{\MRhref}[2]{%
  \href{http://www.ams.org/mathscinet-getitem?mr=#1}{#2}
}
\providecommand{\href}[2]{#2}

\end{document}